\documentclass[english,oneside,reqno]{amsart}

\usepackage[utf8]{inputenc} % input encoding
\usepackage[english]{babel}

\usepackage[T1]{fontenc}

\usepackage{geometry}

\usepackage{xcolor}

\usepackage{todonotes}

\usepackage{graphicx} 				% \includegraphics command
\usepackage{booktabs} 				% for better tables
\usepackage{array}    				% better matrices
\usepackage[parfill]{parskip}		% begin a paragraph with an empty line rather than an indent.
\usepackage{enumitem} 				% environment name: enumerate, augments: \alph  \Alph, \arabic, \roman,  \Roman,

\usepackage{tcolorbox}
\usepackage{nicefrac}
\usepackage{bbm}

\usepackage{tikz}
\usepackage{tikz-cd} % commutative diagrams
\usetikzlibrary{decorations.markings,positioning}
\usetikzlibrary{matrix,decorations.pathmorphing,decorations.pathreplacing,calligraphy,}
\usetikzlibrary{calc, intersections,hobby}
\usetikzlibrary{arrows,shapes.geometric,calc}	% arrow tip library
\usetikzlibrary{bending}		% better arrow head for bended lines
\usepackage{rank-2-roots}
\usepackage{dynkin-diagrams}

\usepackage{mathrsfs}
\usepackage{dutchcal} % dutch mathcal font
\usepackage{amsfonts}
\usepackage{amssymb}

\usepackage{amsmath}

\usepackage{ytableau}

\usepackage{amsthm}
\usepackage{thmtools}

\usepackage{textgreek}
\usepackage{newunicodechar}
\newunicodechar{ϛ}{\ensuremath{\varsigma}} % use varsigma as a stand-in

\usepackage[alphabetic, initials]{amsrefs}

\usepackage{hyperref}
\usepackage[capitalize, nameinlink]{cleveref}

\crefalias{lemma}{lem}

\theoremstyle{plain}
\newtheorem{lem}{Lemma}[section]

\newtheorem{prop}[lem]{Proposition}
\newtheorem{cor}[lem]{Corollary}
\newtheorem{thm}[lem]{Theorem}

\theoremstyle{definition}

\newtheorem{defi}[lem]{Definition}
\newtheorem{rem}[lem]{Remark}

\newtheorem{ex}[lem]{Example}

\newtheorem{conj}[lem]{Conjecture}

\crefname{equation}{}{}

\usepackage[labelfont={sf},font={small},
labelsep=space]{caption}
\calclayout

\newcommand{\mfsl}{\mathfrak{sl}_n(\C)}
\newcommand{\mfg}{\mathfrak{g}}
\newcommand{\Z}{\mathbb{Z}}
\newcommand{\N}{\mathbb{N}}
\newcommand{\C}{\mathbb{C}}
\DeclareMathOperator{\Br}{Br}

\DeclareMathOperator{\SSym}{SSym}

\DeclareMathOperator{\Res}{Res}
\DeclareMathOperator{\End}{End}

\usepackage{t-angles}
\usepackage[all]{xy}
\usepackage{ytableau}
\usepackage[dvipsnames]{xcolor}
\usepackage[enableskew]{youngtab}
\begin{document}
\title[On the center of the BMW algebras]{The center of the BMW algebras and an Okounkov-Vershik like approach.}
\setlength{\headheight}{14.5pt}

\author{\rm Christoforos Milionis}
\address{C.A.M.: School of Mathematics, Statistics and Physics, Newcastle University}
\email{C.A.Milionis2@ncl.ac.uk}

\begin{abstract}
    We use the Jucys-Murphy elements of the Birman-Murakami-Wenzl algebras $B_n(q,t)$, to show that its center over the field of complex numbers for almost all parameters making it semisimple is given by Wheel Laurent polynomials, a subalgebra of the symmetric Laurent polynomials in the JM elements. As an application, we give an Okounkov-Vershik like approach to its finite dimensional representations. In the non-semisimple case related to the type B Lie algebras, the central subalgebra of Wheel Laurent polynomials is large enough to separate blocks of the BMW algebras.
\end{abstract}

\maketitle

\section{Introduction}
\subsection*{Motivation}
\par Let $U(\mfg)=U(\mfsl)$ be the universal enveloping algebra of the simple Lie algebra of type $A_{n-1}$ with its usual triangular decomposition $\mfg=\mathfrak{n}^-\oplus \mathfrak{h}\oplus \mathfrak{n}^+$ and $V=L(\Box)$ be its vector representation, the simple module corresponding to the first fundamental weight. 

\par Since $U(\mfg)$ is a Hopf algebra, the $r$-th fold tensor product $V^{\otimes r}$ is a module over $U(\mfg)$. All polynomial representations of $\mfsl$ appear as direct summands of $V^{\otimes r}$, for $r$ high enough.
\par Hence, the importance of understanding $\End_{U{(\mfg)}}(V^{\otimes r})$ becomes apparent. There is a natural right action of $\C S_r$ on $V^{\otimes r}$ by permuting the tensor factors. We are thus in the following situation:
\begin{equation*}
    U(\mfg)\curvearrowright V^{\otimes r} \curvearrowleft \C S_r.
\end{equation*}
\par Classical Schur-Weyl duality \cite{MR255} shows that the images of these two actions are centralizers of one another. Moreover, decomposing $V^{\otimes r}$ as a $(U(\mfg),\C S_r)$-bimodule allows one to express the irreducible representations of $S_r$ as multiplicity spaces of the corresponding irreducible representation of $\mfsl$ in $V^{\otimes r}$. Since we can do this for any $r\in \N$, there is an interplay between the representations of the chain of algebras 
\[
\C\subset \C S_2 \subset \ldots\subset \C S_{r-1} \subset \C S_r\subset \ldots
\]
and the representations of the Lie algebra $\mfsl$.
\par Since the representations of $\C S_r$ can be understood via $\mfsl$-theory, could we create an approach to the finite dimensional representations of $S_r$ that resembles the highest weight theory approach for Lie algebras?
\par This approach was realized by Okounkov and Vershik in \cite{OV}, where the authors studied the representation theory of the natural chain above simultaneously, building the irreducible representations of $S_n$ by seeing how they decompose when regarded as $S_{n-1}$ representations. 
Their approach utilizes a remarkable family of commuting elements that are compatible with the chain, called Jucys-Murphy elements. These are recursively defined by $x_1=0$ and 
\[
x_{k+1}=s_kx_ks_k+s_k\in \C S_{k+1}
\]
for any $k\geq 1$, where $s_k=(k \; k+1)$ denotes the simple transposition of $S_{k+1}$. It should be noted that $x_k$ is not central, but it centralizes $\C S_{k-1}\subseteq \C S_{k}$ and the algebra they generate $\C[x_1,\ldots,x_k]\subseteq \C S_k$ is a maximal commutative subalgebra, analogous to the Cartan subalgebra in Lie theory. 
An important factor in their approach was the fact that this chain is multiplicity free, meaning that any irreducible representation of $S_n$ regarded as an $S_{n-1}$ representation decomposes in such a way that the multiplicities of the irreducible $S_{n-1}$ representations appearing in the decomposition are equal to one. One then records these so called branching rules in a graph called the branching graph. From this graph, one may construct a unique (up to scalars) basis indexed by paths in the branching graph for every irreducible $S_n$ representation. This basis is called the Gelfand-Zetlin basis, or GZ basis.
\par Jucys in \cite{JucysCenterCSn} used the JM elements to describe the center of the group algebra $Z(\C S_k)$ as the algebra of symmetric polynomials in the JM elements. That is,
\[
Z(\C S_k)=\C[x_1,\ldots,x_k]^{S_k}
\]
%Notably, the family defined above is very compatible with the multiplicity free chain. For every term in the chain, we have that $x_k\in \C S_k$ and, it commutes with every term in the chain before it. In formulas
%\[
%x_k \sigma = \sigma x_k,
%\]
%for all $\sigma \in \C{S_{k-1}}$. 
%\par The GZ-basis of every irreducible $S_n$ representation is a simultaneous eigenvector basis for the family $x_1,\ldots,x_n\in \C S_n$, giving rise to a correspondence between GZ basis elements and vectors containing the corresponding eigenvalues of $x_1,\ldots,x_n$ acting on the GZ basis. Hence, understanding the eigenvalues of the JM elements gives an explicit description of the GZ basis in the branching graph. Okounkov and Vershik used different tools from the classical ones to show that the branching graph is identified with the graph of partitions. Additionally, paths corresponding to GZ basis elements can be identified with standard Young tableau and the eigenvalues of the JM elements correspond to the contents of the boxes in the standard tableaux of the corresponding GZ-basis element.
%\par The Jucys-Murphy of $\C S_n$ have also been used to describe its center $Z(\C S_n)$. In particular, Jucys in  \cite{JucysCenterCSn} showed that the center $Z(\C S_n )$ is equal to the algebra of symmetric polynomials in the first $n$ JM elements. That is, 
%\[
%Z(\C S_n)=\C[x_1,\ldots,x_n]^{S_n}.
%\]
Similar results have been generalized to other types.
\par Replacing $\mfg=\mfsl$ by $\mathfrak{so}(2n+1),\mathfrak{sp}(2n),\mathfrak{so}(2n)$, the centralizer algebra is replaced by a one parameter family of algebras called the Brauer algebras, denoted by $\Br_r(\delta)$ for $\delta\in \Z$ depending on the Lie algebra $\mathfrak{g}$, see \cite{Brauer}. One sees that the relations defining $\Br_r(\delta)$ depend polynomially on $\delta$, hence these can be defined for $\delta\in \C$. These algebras fit again in a natural multiplicity free chain. Analogues of the JM elements were first introduced by Nazarov in \cite{Naz}, where the properties of these elements were extensively studied. Nazarov also determined the center in the semisimple cases to be equal to the subalgebra of symmetric polynomials satisfying an additive cancellation condition evaluated at the JM elements. To be exact, he proved that for $\delta \in \Z_{\geq 0}$ odd or large enough even,
\[
Z(\Br_n(\delta))=\{f\in \C[x_1,\ldots,x_n]^{S_n}: \;f(x_1,-x_1,\ldots,x_n)=f(0,0,x_3,\ldots,x_n)\}.
\]
\par For $\mfg=\mfsl$, if one replaces $V^{\otimes r}$ with $V^{\otimes r}\otimes {(V^{*})}^{\otimes s}$, where $V^{*}$ is the dual representation of $V$, then the centralizer algebra is replaced by a subalgebra of the Brauer algebra called the walled Brauer algebras $B_{r,s}(\delta)$ with $\delta=n$, see for example \cite{BGCMHT}. Even though these algebras do not fit in just one natural mutliplicity free chain, Jung and Kim in \cite{JungKim} used a specific multiplicity free chain and defined analogues of the JM elements for $B_{r,s}(\delta)$, allowing them to describe the center in terms of JM elements.
In particular, they proved that the center is equal to the algebra of super-symmetric polynomials in the JM elements $x_1,\ldots,x_r,y_1,\ldots,y_s$. That is,
\[
Z(B_{r,s}(\delta))=\C[x_1,\ldots,x_r,y_1,\ldots,y_s]^{\SSym}.
\]
\par Of particular interest is the case of the partition algebras $P_{2n}(t)$, first defined by Martin in \cite{MR1265453}. These algebras also admit JM elements as defined in \cite{HaRamPartition}. Recently, Creedon in \cite{CRE} proved that in the semisimple case, the center of the partition algebras is described by the algebra of supersymmetric polynomials in the JM elements,
\[
Z(P_{2n}(t))=\C[x_1,\ldots,x_n,y_1,\ldots,y_n]^{\SSym}.
\]
\par One can address the same questions for quantum groups. For $U_q(\mfsl)$ and $V=L(\omega_1)$ the quantized natural representation, the centralizer algebra of $V^{\otimes r}$ is described by a one parameter deformation of the group algebra of $S_n$, the Hecke algebras $H_n(q)$ while for mixed tensors the centralizer is described by the so called quantized walled Brauer algebras $\mathcal{B}_{r,s}(q,\delta)$.
\subsection*{The BMW algebras}
\par Now if $\mfg\in\{\mathfrak{so}(2n+1),\mathfrak{sp}(2n),\mathfrak{so}(2n)\}$ and $V$ is the quantized vector representation, then the centralizer algebras of the $r-$fold tensor products are described by the Birman-Murakami-Wenzl algebras (BMW for short), $B_n(q,t)$, first introduced in \cite{BirmanWenzl}, see also \cite{WEN}, \cite{MR927059},\cite{LZ}. These algebras are two parameter deformations of the Brauer algebras. The representation theory of this algebras has been extensively studied, see for example \cite{Xi},\cite{RSi1},\cite{RSi2},\cite{RS}. These algebras have a diagrammatic interpretation and also fit in a natural chain 
\[
\C=B_1(q,t)\subset B_2(q,t)\subset \ldots \subset B_n(q,t)\subset \ldots
\]
They are also equipped with a family of JM elements $x_1,\ldots,x_n$ compatible with the above chain, defined by a recursive relation. Hence, they fit into the framework described in all the previous cases above.
\par In this paper, we give a result describing the center of $B_n(q,t)$, which is comparable to the previous cases. One of the main results of this paper is \cref{THM1center}, where we describe all cases in which the center is described by the multiplicative analogue of the algebra describing the center of the Brauer algebras $\Br_n(\delta)$, for $q\in \C$ not a root of unity. That is, we let 
\[
C=\{f\in \C[X_1^{\pm 1},\ldots,X_n^{\pm 1}]^{S_n}: \; f(X_1,X_1^{-1},\ldots,X_n)=f(1,1,\ldots,X_n)\}
\]
and call it the algebra of Wheel Laurent polynomials. Our main motivation came from \cite{DRV}, where the authors computed the center of the affine BMW algebras $aB_n(q,t)$ to be equal to the algebra $C$ in the affine generators $x_1,\ldots,x_n$. We note that their proof is completely algebraic and elementary, using only a basis theorem for $aB_n(q,t)$. 
\subsection*{Main Results}
\par We state the main results of this paper. Let $q,t\in \C$ be such that $q$ is not a root of unity and $WL[x_1,\ldots,x_n]$ be the subalgebra of $B_n(q,t)$, consisting of Wheel Laurent polynomials in the JM elements of $B_n(q,t)$. The first theorem gives a characterization of the center.
\begin{thm}\label{THM1center}
    For $q\in \C$ not a root of unity, the following holds:
    \begin{equation*}
        Z(B_n(q,t))=WL[x_1,\ldots,x_n],
    \end{equation*}
    if and only if we are in one of the following situations:
    \begin{itemize}
        \item $t\neq \pm q^a$, for $a\in \Z$,
        \item $t=\pm q^{2a}$ or $t=\pm q^{2a-1}$, for $\lvert a\rvert \geq n$,
        \item $t=\pm q^{2a}$, with $\lvert a\rvert < n$ such that $B_n(q,t)$ is semisimple.
    \end{itemize}
    In particular, when $t=\pm q^{2a-1}$ with $\lvert a\rvert<n$, it holds that
    \begin{equation*}
        \dim_{\C}WL[x_1,\ldots,x_n]<\dim_{\C}Z(B_n(q,t)),
    \end{equation*}
    with the exception of $n=3$ and $t=q,q^{-1},-q$.
\end{thm}
\par The main ideas of the proof are the following. As a first step, one needs to show that 
\begin{equation*}
    WL[x_1,\ldots,x_n]\leq Z(B_n(q,t)).
\end{equation*}
 This is very easy in our case, as we can exploit the following observation. The affine BMW algebras $aB_n(q,t)$ are essentially an enlargement of $B_n(q,t)$ by freeing the JM elements as affine generators. This translates to having a surjective algebra morphism
\[
aB_n(q,t)\twoheadrightarrow B_n(q,t)
\]
such that the affine generators are sent to the JM elements of $B_n(q,t)$. This morphism restricts to an algebra morphism between the centers. Since the center of $aB_n(q,t)$ was computed in \cite{DRV} to be equal to $C$, and this algebra projects to $WL[x_1,\ldots,x_n]$, the first step is complete.
%One can also check by hand that such polynomials are central by just seeing relations between powers of x_i's and e_i,s_i.
\par Since we are dealing with the cases where $B_n(q,t)$ is semisimple, the Artin-Wedderburn theorem gives us that the dimension of $Z(B_n(q,t))$ is equal to the number of simple $B_n(q,t)$-modules. Additionally, by Schur's lemma, every central element acts on simple modules by a scalar. Hence, we have to show that the algebra $WL[x_1,\ldots,x_n]$ provides the right number of linearly independent elements, by acting on different simple modules by different scalars. In the case of (super)symmetric polynomials, one employs the following strategy. A generating set is given by the elementary (super)symmetric polynomials, a family of polynomials whose definition is given by a generating function. This generating function compactly holds the information of the entire family of generators. One uses the generating function to separate simples, and then uses a point separation argument to conclude that the desired algebra has the correct dimension.
\par In the case of Wheel Laurent polynomials, we define an analogue to the family of elementary (super)symmetric polynomials (see \cref{elemwheel}) and use the technique described above to show that 
\[
\dim_\C WL[x_1,\ldots,x_n]\geq \dim_{\C}Z(B_n(q,t))
\]
which allows us to conclude.
\par We note that interesting combinatorics start emerging when 

\[t=q^{2a} \quad \text{with} \quad a\in \{\lfloor\frac{n}{2}\rfloor-1,\ldots,n-1\},
\]
see \cref{smallevenpowers}.
\par The algebras $B_n(q,t)$ are cellular, see \cite{GL}, \cite{Xi}. In \cite{EnyangMurphyBasisold}, Enyang proved that there is a Murphy basis for every cell module of $B_n(q,t)$. In the cases where $B_n(q,t)$ is semisimple, we use this result to find the existence of a simultaneous eigenvector basis of every simple $B_n(q,t)$ module with respect to the action of the JM elements $x_1,\ldots,x_n$. We also call this basis the Murphy basis. One naturally asks if this basis is the same as the GZ-basis. Denote the Gelfand-Zetlin algebra of $B_n(q,t)$ by $GZ(n)$ (see \cref{GZalgebra}).
\begin{thm}
It holds that
\begin{equation*}
    GZ(n)=\C[x_1^{\pm 1},\ldots,x_n^{\pm 1}],
\end{equation*}
if and only if $q,t\in \C$ are as in \cref{THM1center}.
\par In these cases, the Murphy basis is equal (up to scalars) to the GZ basis, for every simple module of $B_n(q,t)$.
\end{thm}
\par Now suppose $t=q^{2a}$ such that $B_n(q,t)$ is not semisimple. The proof that $WL[x_1,\ldots,x_n]\leq Z(B_n(q,t))$ is independent of the semisimplicity of $B_n(q,t)$, hence it holds for any $q,t\in \C$.We can show that the algebra $WL[x_1,\ldots,x_n]$ is still large enough to recognize the block structure of $B_n(q,t)$. The blocks of $B_n(q,t)$ were studied in \cite{RSi2}, where the authors gave combinatorial conditions describing when two simple modules corresponding to two partitions $(\lambda,f),(\mu,l)$ lie in the same block (see \cref{admissibility condition}). Our next result yields an equivalent characterization for blocks.
\begin{thm}
    Let $t=q^{2a}$ for $a\in \Z$. Two simple $B_n(q,t)$-modules $L(\lambda,f),L(\mu,l)$ lie in the same block if and only if 
    \begin{equation*}
            W(\lambda,t)=W(\mu,t).
    \end{equation*}    
\end{thm}

\section{The Birman-Murakami-Wenzl algebras} \label{Section1}
In this section, we recall the definition of the BMW algebras, together with some well-known facts about their representation theory. The BMW algebras are the $q$-analogues of the Brauer algebras. For details, see for example \cite{Mor}.
\par Denote by $A=\Z[q^{\pm 1},t^{\pm 1},(q-q^{-1})^{-1}]$ and $R=\Z[t^{\pm 1},(q-q^{-1}),\delta]/\langle t-t^{-1}=(q-q^{-1})(\delta-1)\rangle$.
    \begin{defi}\label{BMWdef}\cite{BirmanWenzl}
    For $n\in{\mathbb{N}}$, the Birman-Murakami-Wenzl algebra $B_n(q,t)$ is the quotient of the free $A$-algebra generated by invertible elements $s_1,\ldots,s_{n-1}$ modulo the ideal generated by the following relations:
    \begin{align*}
        & \text{(Spectrum)} \;  (s_i-q)(s_i+q^{-1})(s_i-t^{-1})=0, \; \text{for} \; 1\leq i \leq n-1, \\
        &\text{(Braid relation)} \;  s_is_{i+1}s_i=s_{i+1}s_is_{i+1}, \; \text{for} \; 1\leq i \leq n-2, \\
        & \text{(Locality)} \; s_is_j=s_js_i, \; \text{for} \; \lvert i-j\rvert>1, \\
        &\text{(Delooping 1)} \;  e_is^{\pm}_{j}e_i = t^{\pm 1} e_i, \; \text{for} \; 1\leq i \leq n-1 \; \text{and} \; j=i\pm1, \\
        & \text{(Delooping 2)} \;e_is_i=s_ie_i = t^{-1} e_i, \; \text{for} \; 1\leq i \leq n-1, \\
        &  \text{(Snake)} \; e_ie_{i\pm 1}e_i=e_i, \; \text{for} \; 1\leq i\leq n-1, %\\
       % & s_is_{i\pm1}e_i=e_{i\pm1}e_i \; \text{for} \; 1\leq i \leq n-1,
    \end{align*}
    where \begin{equation*}
        \text{(Kauffman-Skein)} \;  e_i=1-(q-q^{-1})^{-1}(s_i-s^{-1}_i), \; \text{for} \; 1\leq i \leq n-1.
    \end{equation*}
\end{defi}
\begin{rem}
    In fact, the snake relations over $A$ can be proved from the other relations. We include it in the presentation as it is a relation that is commonly used and has a categorical interpretation. Note also that if we defined $B_n(q,t)$ over $R$, we would have to add the $e_i$'s as additional generators. It follows from the relations above that $e_i^2=\delta e_i$, for all $1\leq i \leq n-1$, where $\delta=\frac{t-t^{-1}}{q-q^{-1}}+1$.
\end{rem}
\par Morton and Wassermann in \cite{Mor} proved that the BMW algebras are free over $A$ (or $R$), with a basis given by lifts of Brauer diagrams. 
Hence, $\operatorname{rank}_A(B_n(q,t))=\operatorname{rank}_R(B_n(q,t))=(2n-1)!!=1\cdot 3\cdots(2n-1)$. They actually proved that the BMW algebras are isomorphic to the Kauffman tangle algebras. Hence, one can employ the graphical calculus of the Kauffman tangle algebras when studying $B_n(q,t)$. Using this, they proved the following. 
\begin{thm}\cite{Mor}
    Let $\sigma:A\rightarrow \Z[\delta]$ be the algebra morphism defined by extending $\sigma(q-q^{-1})=0, \; \sigma(t)=1, \; \sigma(\delta)=x$ and $B_{n,\Z[x]}(q,t)=B_n(q,t)\otimes_{R} \Z[x]$. Then, as $\Z[x]$-algebras,
    \[
    B_{n,\Z[x]}(q,t)\cong \Br_n(x).
    \]
    where $\Br_n(x)$ is the Brauer algebra over $\Z[x]$.
\end{thm}
Since eventually we will work with $B_n(q,t)$ over $\C$ by specializing $q,t$ and compare the situation with that of the Brauer algebras, we make some useful remarks. For $n\in \Z$, denote by \[
[n]_q=\frac{q^n-q^{-n}}{q-q^{-1}}=q^{n-1}+q^{n-3}+\ldots+q^{-(n-3)}+q^{-(n-1)}\] the quantum integer corresponding to $n$. For $q\in \C$ such that $q-q^{-1}\in \C^{*}$ and $t=\pm q^N$, for some integer $N$, then we have $\delta=[\pm N]_q+1$. Specializing further $q\rightarrow 1$, we get $\delta=\pm N+1$. Hence, the classical limit of $B_n(q,\pm q^N)$ is $\Br_n(\pm N+1)$, which gives us a parity offset in the sense that when $t$ is an even (odd) power of $q$, this corresponds to the case of Brauer with odd (even) parameter.
\par Using this graphical calculus and the basis theorem in \cite{Mor}, it is now easy to show that the subalgebra of $B_n(q,t)$ generated by $s_1,\ldots,s_{n-2}$ is isomorphic to $B_{n-1}(q,t)$. This means that we have a tower of $A$ (respectively $R$)-algebras,
\[
B_1(q,t)\hookrightarrow B_2(q,t) \hookrightarrow \ldots\hookrightarrow B_{n-1}(q,t) \hookrightarrow B_n(q,t)\hookrightarrow\ldots
\]
%Alternatively, if one wants to see that this is an algebra chain, use Markov trace: The inclusion i_n:B_{n-1}\rightarrow B_n is an algebra morphism and if we define the normalized markov trace as p_n:B_n \rightarrow B_{n-1} p_n(x)=\frac{1}{\delta}x' where e_nxe_n = x' e_n for unique x' \in B_{n-1}, then we see that p_n\circ i_n = id_{B_{n-1}} and so i_n is injective!d
\section{Cellular structure and representation theory}
\par In this section, we recall some well-known results concerning the cellular structure of $B_n(q,t)$ and its representation theory.
\par It was proved by Xi in \cite{Xi}, that $B_n(q,t)$ is a cellular algebra over $A$ in the sense of Graham and Lehrer \cite{GL}. Later, Enyang in \cite{EnyangMurphyBasisold} constructed a Murphy basis $\{m_T \; | \; T\in T_n^{\operatorname{ud}}(\lambda)\}$ for each cell module of $B_n(q,t)$ over $A$. To make this precise, we need a remarkable family of commuting elements in $B_n(q,t)$ called Jucys-Murphy elements.
\par For any algebra morphism $\operatorname{ev_\C}:A\rightarrow \C$, the Birman-Murakami-Wenzl algebras over $\C$ are defined as $B_{n,\C}(q,t)=B_n(q,t)\otimes_A \C$, where we abuse notation and denote the images of $q,t$ via $\operatorname{ev_\C}$ by $q,t\in \C^{*}$, for $q^2\neq 1$. Whenever it is clear from the context, we will denote $B_{n,\C}(q,t)$ by $B_n(q,t)$.
\begin{defi}
    The Hecke algebra $H_n(q)$ associated to the symmetric group $S_n$ is the associative algebra over $\Z[q^{\pm 1}]$ generated by invertible $s_1,\ldots, s_{n-1}$ modulo the ideal generated by the relations:
    \begin{align*}
        & \; (s_i-q)(s_i+q^{-1})=0, \; \text{for} \; 1\leq i \leq n-1, \\
        &  \; s_is_{i+1}s_i=s_{i+1}s_is_{i+1}, \; \text{for} \; 1\leq i \leq n-2 \\
        & \; s_is_j=s_js_i, \; \text{for} \; \lvert i-j\rvert > 1.
    \end{align*}
\end{defi}
It is well known that $B_n(q,t)/(e_1)\cong H_n(q)$ as $\Z[q^{\pm 1}]$-algebras, where $(e_1)$ is the ideal generated by the quasi-idempotent $e_1$. The corresponding isomorphism maps $s_i+ (e_1)$ to the generator with the same name in $H_n(q)$. However, $H_n(q)$ is not a subalgebra of $B_n(q,t)$.
\par We now recall some combinatorics needed to state Enyang's result.
\par A partition of $n$ is a weakly decreasing sequence of non-negative integers $\lambda=(\lambda_1,\lambda_2,\ldots)$ such that $\lambda_1\geq \lambda_2\geq \ldots$, where $\lvert\lambda\rvert:=\sum\limits_{i=1}^{\infty} \lambda_i=n$. We use the notation $\lambda\vdash n$ when $\lambda$ is a partition of $n$. Let $\mathcal{P}(n)$ be the set of all partitions of $n\in \Z_{\geq 0}$. This is a poset with dominance order $\unrhd$ on it. For two partitions $\lambda,\mu\in \mathcal{P}(n)$, we say that $\lambda \unrhd \mu$ if and only if $\sum\limits_{i=1}^j \lambda_i\geq \sum\limits_{i=1}^j \mu_i$, for all $j$. Write $\lambda \rhd \mu$ if $\lambda \unrhd \mu$ and $\lambda\neq \mu$.
\par Partitions will be depicted as Young diagrams. The Young diagram of a partition $\lambda=(\lambda_1,\lambda_2,\ldots)$ is a collection of boxes arranged in left aligned rows with $\lambda_i$ boxes in the $i$-th row. We identify a partition with its corresponding Young diagram.
A standard tableau of shape $\lambda\vdash n$ is obtained by inserting the numbers $1,\ldots,n$ to the boxes of $\lambda$ such that the numbers are strictly increasing, both along the rows and the columns of $\lambda$.
\begin{ex}\label{stdtab32}
    The following are standard tableaux of shape $(3,2)\vdash 5$
   \begin{align*}
  &  \scalebox{0.7}{\begin{ytableau}
        1 & 2 & 3 \\
        4 & 5
    \end{ytableau}
     \quad 
     \begin{ytableau}
        1 & 3 & 5 \\
        2 & 4
    \end{ytableau} }
    \end{align*}
\end{ex}
A box in position $(i,j)$ is called removable for $\lambda\vdash n$ (respectively, addable) if by removing this box the result is again a partition $\mu \vdash n-1$ (by adding this box we get a partition of $n+1$). These correspond to right inner, respectively outer corner boxes of $\lambda$. Denote the set of removable boxes of $\lambda$ by $R(\lambda)$ (respectively, $A(\lambda)$ for the set of addable boxes).
\begin{ex}
    The removable boxes of the partition $\lambda=(3,2)\vdash 5$ are the boxes in positions $(1,3),(2,2)$, while the addable boxes are in positions $(1,4),(2,3),(3,1)$. 
    The removable boxes are colored red, while the addable ones are colored by blue.
    \[\scalebox{0.7}
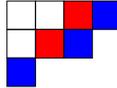
{\begin{ytableau}
       \ & \ & *(red) & *(blue) \\
        & *(red) & *(blue) \\
        *(blue)
    \end{ytableau}}\]
\end{ex}
\par Denote by $T^{\operatorname{std}}(\lambda)$ the set of all $\lambda$ standard tableaux. For a partition $\lambda \vdash n$ and a box $p=(i,j)$, call the number $j-i$ the content of that box and denote it by $c(p)=j-i$. Contents separate the Young diagram of $\lambda$ in diagonals. If we denote by $h(\lambda)$ (respectively $W(\lambda,t)$) the height (respectively width) of $\lambda$, i.e., the number of rows -1 (respectively the number of columns -1), then the contents appearing in $\lambda \vdash n$ are the numbers in the set $\{-h(\lambda),\ldots,w(\lambda)\}$.
\par Let $D(\lambda)=\{-h(\lambda),\ldots,w(\lambda)\}$ be the set of contents of $\lambda$ and for $i\in D(\lambda)$, let $m_{\lambda}(i)$ be the length of the diagonal with content $i$.  It is easy to see that contents, together with the length of their diagonals, completely determine the shape of the partition. That is, if $\lambda,\mu\vdash n$ are such that $D(\lambda)=D(\mu)$ and $m_{\lambda}(i)=m_{\mu}(i)$ for all $i\in D(\lambda)$, then $\lambda = \mu$.
\par By ordering partitions via inclusion of their corresponding Young diagrams, we get a poset, called the Young poset. This is the leveled graph with vertex set at level $n$ being $\mathcal{P}(n)$, and $\lambda\vdash n$ is connected with an edge to $\mu \vdash n+1$ and write $\lambda \rightarrow \mu$, if $\mu$ can be obtained from $\lambda$ by adding a box (equivalently, $\lambda$ can be obtained by removing a box from $\mu$). 
\par A path to $\lambda\vdash n$ in this graph is an $n+1$-tuple $T=(T_i)_{i=0}^n$, such that $T_0=\emptyset \vdash 0, \; T_n=\lambda, \; T_i\vdash i$ for all $i=1,\ldots,n$ such that $T_i\rightarrow T_{i+1}$. It is well known that the set of paths to $\lambda$ is in bijection with standard tableaux of shape $\lambda$, for any partition.
\begin{ex}\label{canpath}
    The standard tableau of shape $(3,2)\vdash 5$ 
    \[\scalebox{0.7}{
    \begin{ytableau}
        1 & 3 & 5 \\
        2 & 4
    \end{ytableau}
    }\]
    corresponds to the following path to $(3,2)$
    \[
\ytableausetup{boxsize=0.7em}
\vcenter{\hbox{$\emptyset$}} \rightarrow
\vcenter{\hbox{\ydiagram{1}}} \rightarrow
\vcenter{\hbox{\ydiagram{1,1}}} \rightarrow
\vcenter{\hbox{\ydiagram{2,1}}} \rightarrow
\vcenter{\hbox{\ydiagram{2,2}}} \rightarrow
\vcenter{\hbox{\ydiagram{3,2}}}
\]
\end{ex}
Denote by $T^{\operatorname{can}}(\lambda)$ the canonical path to $\lambda$, that is, the path that prioritizes completing the rows of $\lambda$ before moving on to adding boxes in the next row.
\begin{ex}
    The canonical path to $(3,2)\vdash 5$ is the path corresponding to the first standard tableau in \cref{stdtab32}.
\end{ex}
Note that we can give the set of paths to $\lambda$ the structure of a poset, by setting $T\leq S$ if $T_i \unlhd S_i$, for all $i=0,\ldots,n$. Then the path $T^{\operatorname{can}}(\lambda)$ is maximal among the paths to $\lambda$.
\subsection{Combinatorics for BMW algebras and Jucys-Murphy elements.}
\par Now let $\Lambda_n=\{ (\lambda,f) \; | \; \lambda\in \mathcal{P}(n-2f), \; f=0,\ldots,[\frac{n}{2}]\}$. It is a poset with $\succeq$ defined on it. Explicitly, for $(\lambda,f),(\mu,l)\in \Lambda_n$, $(\lambda,f)\succeq (\mu,l)$ is valid if $f>l$ or $f=l$ and $\lambda \unrhd \mu$ in the dominance order $\unrhd$. Write $(\lambda,f)\succ (\mu,l)$, if $(\lambda,f)\succeq (\mu,l)$ and $(\lambda,f)\neq (\mu,l)$.
\par Given $(\lambda,f)\in \Lambda_n$, an updown $\lambda$-tableau is a sequence of partitions $T=(T_i)_{i=0}^n$ such that $T_0=\emptyset, \; T_n=\lambda$ and either $T_i\rightarrow T_{i-1}$ or $T_{i-1}\rightarrow T_i$ in the Young graph. This means that each partition is obtained from the previous one by either adding or removing a box.
\begin{ex}
    Consider $(1,(2,2))\in \Lambda_6$. The following is an updown $(2,2)$-tableau
    \[
\vcenter{\hbox{$\emptyset$}} \rightarrow
\vcenter{\hbox{\ydiagram{1}}} \rightarrow
\vcenter{\hbox{$\emptyset$}} \rightarrow
\vcenter{\hbox{\ydiagram{1}}} \rightarrow
\vcenter{\hbox{\ydiagram{2}}} \rightarrow
\vcenter{\hbox{\ydiagram{2,1}}} \rightarrow
\vcenter{\hbox{\ydiagram{2,2}}}
\]

\end{ex}
Denote the set of updown $\lambda$-tableaux by $T^{\operatorname{ud}}_n(\lambda)$. Note that when $\lambda \vdash n$, then $T^{\operatorname{ud}}_n(\lambda)$ is identified with the set of all $\lambda$-standard tableaux. Rui and Si defined a partial order $\succeq$ on $T^{\operatorname{ud}}_n(\lambda)$ as follows. For $S,T\in T^{\operatorname{ud}}_n(\lambda)$ we write $S\overset{k}\succ T$ if $S_k\rhd T_k$ and $S_j=T_j$ for all $k+1\leq j \leq n$. If there is $1\leq k \leq n-1$ such that $S\overset{k}\succ T$, we write $S\succ T$. 

\begin{defi}\label{jmbmw}
    Let $x_1=t$ and $x_{i+1}=s_ix_is_i$, for $i=1,\ldots,n-1$. Call the elements $x_1,\ldots,x_n\in B_n(q,t)$ the Jucys-Murphy elements of the BMW algebras.
\end{defi}
Note that under the embedding $B_i(q,t)\hookrightarrow B_n(q,t)$, we have $x_i\in B_i(q,t)$ for all $i=1,\ldots,n$. Further, it is easy to see that
\[
x_i\in C(B_i(q,t),B_{i-1}(q,t))=\{f\in B_i(q,t) \; | \; fx=xf, \quad \forall x\in B_{i-1}(q,t)\},
\]
that is, $x_i$ centralizes the subalgebra $B_{i-1}(q,t)$. This proves that the subalgebra generated by $x_1^{\pm 1},x_2^{\pm 1},\ldots,x_n^{\pm}$ is commutative. Further, $x_1\ldots x_n$ is central in $B_n(q,t)$ over $A$.
\begin{defi}\label{qtcont}
  For any $T\in T^{\operatorname{ud}}_n(\lambda)$ with $(\lambda,f)\in \Lambda_n$, define the $(q,t)$-content of $T_k$ to be $c_T(k)\in A$ given by
  \begin{align*}
      c_T(k)=\begin{cases}
          tq^{2i}, \quad \text{if} \; T_k=T_{k-1}\cup \{p\} \\
          t^{-1}q^{-2i}, \quad \text{if} \; T_{k-1}=T_k\cup \{p\},
      \end{cases}
  \end{align*}
  where $c(p)=i\in D(\lambda)$.
\end{defi}
\ytableausetup{boxsize=0.7em}
\begin{ex}
    Let $T=\Big(
\vcenter{\hbox{$\emptyset$}} \rightarrow
\vcenter{\hbox{\ydiagram{1}}} \rightarrow
\vcenter{\hbox{\ydiagram{2}}} \rightarrow
\vcenter{\hbox{\ydiagram{1}}} \rightarrow
\vcenter{\hbox{\ydiagram{1,1}}}
\Big)$. Then $c_T(1)=t,c_T(2)=tq^2,c_T(3)=t^{-1}q^{-2},c_T(4)=tq^{-2}$.
\end{ex}
\begin{rem}
    \cref{qtcont} can be simplified as follows. To each step of an updown $\lambda$-tableau, we record the number $tq^{2i}$ if we added a box in the $i$-th diagonal for that step, and we record the number $t^{-1}q^{-2i}$ if we removed a box of content $i$ in that step. This is a way to keep track of the up-down steps in a path in the same way that contents remember standard tableaux.
    \end{rem}
\subsection{Semisimplicity and Branching graph}
In this subsection we recall some results about the representation theory of $B_n(q,t)$ over $\C$.
\begin{thm}\cite{WEN}\label{semibywen}
    For $q,t\in \C$, the algebra $B_n(q,t)$ is semisimple, unless $q$ is a root of unity or $t=\pm q^a$, for some $a\in \Z$.
\end{thm}
Throughout this section, we assume that $q,t\in \C$ are such that $q$ is not a root of unity and $t\neq \pm q^a$, for $a\in \Z$. Then the algebra $B_n(q,t)$ is semisimple. When $q,t\in \C$ are as above, we will refer to them as generic parameters for the BMW algebras. Intuitively this means that we are working with $B_n(q,t)$ defined over $\C(q,t)$ and then specializing to $\C$.
\begin{thm}
    The simple modules of $B_n(q,t)$ are indexed by the set $\Lambda_n$.
\end{thm}
% Simple proof for generic parameters is that for e_1 B_n/(e_1) is H_n and the corner algebra e_1B_ne_1 is B_{n-2}, so an induction yields the result quickly.
From now on, denote the simple module of $B_n(q,t)$ corresponding to $(\lambda,f)\in \Lambda_n$ by $L(\lambda)$ when it is clear that $\lambda \vdash n-2f$. Since we established that $B_{n-1}(q,t)$ is a subalgebra of $B_n(q,t)$, we can restrict simple modules of $B_n(q,t)$ to $B_{n-1}(q,t)$ and see what the decomposition rules are.
\begin{thm}\cite{WEN}\label{branching}
    Let $L(\lambda,f)$ be the simple module of $B_n(q,t)$ corresponding to $(\lambda,f)\in \Lambda_n$. Then
    \begin{align*}
        \Res^{B_n(q,t)}_{B_{n-1}(q,t)}L(\lambda)=\bigoplus\limits_{\Lambda_{n-1}\ni\mu=\lambda\pm \Box}L(\mu),
    \end{align*}
    or in simple words, the restriction decomposes into the irreducible $B_{n-1}(q,t)$ representations corresponding to either removing or adding a box to $\lambda$. In particular, the above decomposition is multiplicity free. 
\end{thm}
\begin{ex}
    Let $n=5$, $(1,(2,1))\in \Lambda_5$. Then as a $B_{4}(q,t)$ representation it decomposes as
    \[
    \Res_{B_4(q,t)}^{B_5(q,t)} L((2,1))=L((2,2))\oplus L((3,1)) \oplus L((2,1,1))\oplus L((1,1))\oplus L((2)).
    \]
\end{ex}
\par By \cref{semibywen} and \cref{branching}, we have a tower of semisimple $\C$-algebras
\[
B_1(q,t)\subseteq B_2(q,t)\subseteq \ldots \subseteq B_{n-1}(q,t)\subseteq B_n(q,t)\subseteq\ldots
\]
of which the branching graph is defined as follows. It is the leveled graph with vertex set at level $n$ being $\Lambda_n$ and two vertices $(\mu,l)\in \Lambda_{n-1},(\lambda,f)\in \Lambda_n$ are connected by an edge if and only if $\mu=\lambda\pm \Box$. Paths to $(\lambda,f)\in \Lambda_n$ are given by $T^{\operatorname{ud}}_n(\lambda)$. For every irreducible $B_n(q,t)$ module, there is a canonical basis obtained by iterative restrictions of this module to $B_1(q,t)\cong \C$, for which all irreducible modules are one dimensional. For $\lambda\vdash n-2f$, this basis is indexed by paths $T\in T^{\operatorname{ud}}_n(\lambda)$. Denote it by $\{v_T \; | \; T\in T^{\operatorname{ud}}_n(\lambda)\}$ and call it the Gelfand-Zeitlin basis for $L(\lambda)$.
\begin{ex}
    The first 4 levels of the above branching graph are displayed below.
    \ytableausetup{boxsize=0.8em}
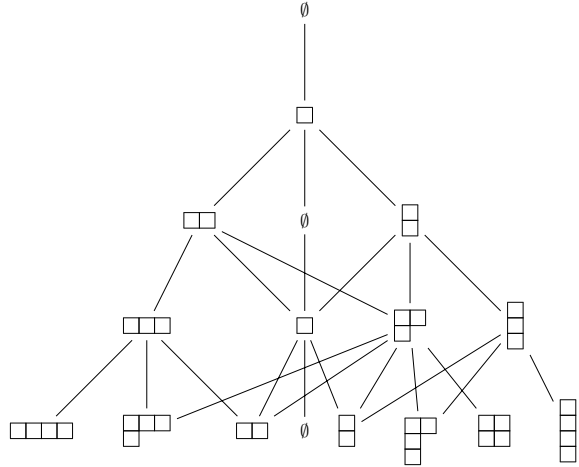
\begin{center}\scalebox{0.7}{
\begin{tikzpicture}[yscale=2, xscale=2]
% nodes
\node (0) at (0,0) {$\emptyset$};
\node (1) at (0,-1) {$\ydiagram{1}$};
\node (11) at (1,-2) {$\ydiagram{1,1}$};
\node (111) at (2,-3) {$\ydiagram{1,1,1}$};
\node (1111) at (2.5,-4) {$\ydiagram{1,1,1,1}$};
\node (2) at (-1,-2) {$\ydiagram{2}$};
\node(00) at (0,-2) {$\emptyset$};
\node (001) at (0,-3) {$\ydiagram{1}$};
\node (21) at (1,-3) {$\ydiagram{2,1}$};
\node (3) at (-1.5,-3) {$\ydiagram{3}$};
\node (211) at (1.1,-4.1) {$\ydiagram{2,1,1}$};
\node (31) at (-1.5,-4) {$\ydiagram{3,1}$};
\node (4) at (-2.5,-4) {$\ydiagram{4}$};
\node (22) at (1.8,-4) {$\ydiagram{2,2}$};
\node (0002) at (-0.5,-4) {$\ydiagram{2}$};
\node (00011) at (0.4,-4) {$\ydiagram{1,1}$};
\node (0000) at (0,-4) {$\emptyset$};
% edges
\draw
 (0) to (1) 
 (1) to (11)
 (1) to (00)
 (1) to (2)
 (00) to (001)
 (001) to (0002)
 (001) to (00011)
 (001) to (0000)
 (3) to (0002)
 (111) to (00011)
 (21) to (00011)
 (21) to (0002)
 (2) to (001)
 (2) to (3)
 (2) to (21)
 (3) to (4)
 (3) to (31)
 (11) to (001)
 (11) to (21)
 (11) to (111)
 (21) to (31)
 (21) to (211)
 (21) to (22)
 (111) to (211)
 (111) to (1111);
\end{tikzpicture}}
\end{center}
\end{ex}
\begin{defi} \label{GZalgebra}
    The Gelfand-Zeitlin algebra of the multiplicity free tower of algebras
    \[
    B_1(q,t)\subseteq B_2(q,t)\subseteq \ldots \subseteq B_{n-1}(q,t)\subseteq B_n(q,t)\subseteq\ldots
    \]
    is the algebra 
    \[
    GZ(n)=\langle Z(1),\ldots,Z(n)\rangle
    \]
    where $Z(i)=Z(B_i(q,t))$ is the center of $B_i(q,t)$.
\end{defi}
We make some observations for $GZ(n)$. It is a commutative subalgebra of $B_n(q,t)$ and since the tower is multiplicity free, it is a maximal commutative subalgebra. Further, using the Artin-Wedderburn theorem there is an equivalent characterization given as follows.
\begin{prop}\label{eqcharofgz}
The Gelfand-Zeitlin algebra is characterized as follows
    \[
    GZ(n)=\{f\in B_n(q,t) \; : fv_T\in \C v_T, \;  T\in T^{\operatorname{ud}}_n(\lambda), \;  (\lambda,f)\in \Lambda_n\}.
    \]
    That is, it consists of all elements such that in every simple module, the GZ-basis is a simultaneous eigenvector basis. Further, the GZ-basis elements are uniquely determined (up to scalars) by the corresponding eigenvalues.
\end{prop}
The algebras $GZ(i)$ also form a tower of algebras. We now show that the Jucys Murphy elements of $B_n(q,t)$ together with their inverses defined in \cref{jmbmw} are in $GZ(n)$. This will follow from the following general observation.
% Easier proof for this lemma. It is clear that x_1\cdots x_k is in Z(k) and (x_1\cdots x_{k-1})^{-1} is in Z(k-1) so x_k is in GZ(n) as the product of the two. One can always use such a trick, depending on if the family of JM elements is additive or multiplicative.
\begin{lem}\label{centralizerinGZ}
    \[
    C(B_n(q,t),B_{n-1}(q,t))\subseteq GZ(n)
    \]
    for all $n$.
\end{lem}
\begin{proof}
    We use the equivalent characterization of the GZ algebra given in \cref{eqcharofgz}. Let $L(\lambda)$ be the simple corresponding to $\lambda\vdash n-2f$ and $x\in B_n(q,t)$ such that $xf=fx$, for any $f\in B_{n-1}(q,t)$. Hence, we get that
    \[
    x\in \End_{B_{n-1}(q,t)}(\Res_{B_{n-1}(q,t)}^{B_n(q,t)}(L(\lambda))).
    \]
    Using \cref{branching} and Schur's lemma, we get
    \[
    x\in \bigoplus\limits_{\mu=\lambda\pm \Box} \End_{B_{n-1}(q,t)}(L(\mu))\cong \bigoplus\limits_{\mu=\lambda \pm \Box} \C_{\mu}.
    \]
    For $T=(T_i)_{i=0}^n\in T_n^{\operatorname{ud}}(\lambda)$, we have $v_T\in L(T_{n-1})$. Since $x$ acts on any simple $B_{n-1}(q,t)$ module appearing in the decomposition of $L(\lambda)$ as a scalar, we get that $xv_T\in \C v_T$, proving that $x$ is diagonal in the GZ basis. Hence $x\in GZ(n)$.
\end{proof}
We have already observed that the Jucys Murphy elements satisfy 
\[
x_i\in C(B_i(q,t),B_{i-1}(q,t)),\]
for all $i=1,\ldots,n$. 
\begin{cor}\label{jminsideGZ}
    The algebra generated by the Jucys-Murphy elements together with their inverses is a subalgebra of the Gelfand-Zeitlin algebra
    \begin{equation*}
        \langle x_1^{\pm 1},\ldots,x_n^{\pm 1}\rangle\subseteq GZ(n).
    \end{equation*}
\end{cor}
\begin{rem}
    An easier proof of \cref{jminsideGZ} is given by the following simpler observation. For any $k\leq n$, the element $x_1\cdots x_k\in Z(k)$ and $(x_1\cdots x_{k-1})^{-1}\in Z(k-1)$, so their product $x_k\in GZ(n)=\langle Z(1),\ldots,Z(n)\rangle$. One can always use this trick with multiplicative or additive families of JM elements.
\end{rem}
This implies that for every simple $B_n(q,t)$-module $L(\lambda)$, there is a simultaneous eigenvector basis $\{v_T \; | \; T\in T_n^{\operatorname{ud}}(\lambda)\}$ of $x_1,\ldots,x_n$. We aim to show that the eigenvalues of $x_1,\ldots,x_n$ in this basis uniquely characterize $L(\lambda)$. First we have to identify the possible eigenvalues of the $x_i$'s in $L(\lambda)$. To do so, we use cellular theory. 
\subsection{Spectrum of Jucys-Murphy elements.}
\par Enyang in \cite{EnyangMurphyBasisold} constructed a Murphy basis $\{m_T \; | \; T\in T_n^{\operatorname{ud}}(\lambda)\}$ for every cell module $\Delta(\lambda,f)=L(\lambda)$ of $B_n(q,t)$. The action of $x_i$ in this basis is given by an upper triangular matrix with diagonal entries $\{c_T(i) \; | \; T\in T^{\operatorname{ud}}_n(\lambda)\}$, where $c_T(i)$ is as in \cref{qtcont}. This implies that the eigenvalues of $x_i$ are exactly all $c_T(i)$, for $T\in T^{\operatorname{ud}}_n(\lambda)$. Since the family $x_1^{\pm 1},\ldots,x_n^{\pm 1}$ is simultaneously diagonalizable, there exists a simultaneous eigenvector basis $\{f_T \; | \; T\in T^{\operatorname{ud}}_n(\lambda)\}$ of $\Delta(\lambda,f)=L(\lambda)$, such that 
\[
x_if_T=c_T(i)f_T,
\]
for every $i=1,\ldots,n$ and $T\in T_n^{\operatorname{ud}}(\lambda)$. We also call the basis $\{f_T \; |\; T\in T_n^{\operatorname{ud}}(\lambda)\}$ the Murphy basis of $L(\lambda)$.
%One can avoid cellularity arguments by using SW duality probably.
\begin{lem}\label{JMactseminormal}
    In the basis $\{f_T \; | \; T\in T^{\operatorname{ud}}_n(\lambda)\}$ described above, the action of the Jucys-Murphy elements is given by
    \begin{equation*}
        x_if_T=c_T(i)f_T \; , \quad i=1,\ldots,n.
    \end{equation*}
\end{lem}
\begin{rem}
    It is worth noting that we do not yet know if the basis $\{f_T \; | \; T\in T_n^{\operatorname{ud}}(\lambda)\}$ is the GZ basis for $L(\lambda)$. That is because, even though with respect to this basis the algebra generated by the Jucys-Murphy elements together with their inverses is simultaneously diagonalized, this does not mean that all elements of $GZ(n)$ are simultaneously diagonalized with respect to $\{f_T \; | \; T\in T^{\operatorname{ud}}_n(\lambda)\}$.
\end{rem}
\section{The Affine Birman-Murakami-Wenzl algebras.}
\par Even though the Jucys-Murphy elements hold important information about the structure of $B_n(q,t)$, the algebra does not naturally see them as special elements there. That is, they are not naturally in a basis of $B_n(q,t)$, or, the diagrammatics of $B_n(q,t)$ do not include them as special diagrams.
\par Suppose $q,t\in \C$ are generic. In \cite{DRV}, the authors work with the affine version of $B_n(q,t)$, an algebra which can be thought of as an enlargement of $B_n(q,t)$ such that its diagrammatics naturally include the Jucys-Murphy elements as new affine generators. For more information on a categorical and diagrammatic intuition for this algebra, see \cite{MYSA}.
\begin{defi}\cite{DRV}\label{afbmwdef}
    The $n$-th affine BMW algebra $aB_n(q,t)$ is the algebra generated by invertible $s_1,\ldots,s_{n-1}$ (corresponding to positive crossings) modulo the BMW relations, plus an affine generator $x$ such that if $x_1=tx$ and $x_{i+1}=s_ix_is_i$, for $i=1,\ldots,n-1$, the following relations hold:
    \begin{align*}
        & x_1x_2=x_2x_1, \\
        & [x,s_i]=0 \quad \forall i\geq 2, \\
        & x_ix_{i+1}s_i=s_ix_ix_{i+1}, \; \forall i\in\{1,\ldots,n-1\}, \\
        & x_ix_{i+1}e_i=e_ix_ix_{i+1}, \; \forall i\in \{1,\ldots,n-1\}, \\
        & e_1x_1^le_1=Z_0^{(l)}e_1, 
    \end{align*}
    where for $l\in \Z$, $Z_0^{(l)}\in \C$.
\end{defi}
\begin{rem}
   Some of the above relations might be redundant, like the commutation with $s_i$, but we include the relation to be as close as possible to the diagrammatic intuition of \cite{MYSA}.
\end{rem}
\begin{rem}
    One should think of $B_n(q,t)$ as the algebra occurring via evaluating the affine generators of $aB_n(q,t)$ to the Jucys-Murphy elements of $B_n(q,t)$. While this is not automatically the case when we restrict to subalgebras of $aB_n(q,t)$, one should naturally ask the question if this is indeed the case for the center.
\end{rem}
Taking advantage of a basis theorem for $aB_n(q,t)$ proved in \cite{GH}, the authors of \cite{DRV} calculated its center algebraically. Note that we present their result over $\C$, but their result holds more generally for any commutative ring satisfying certain admissibility conditions.
\begin{thm}\cite{DRV}\label{cenafbmw}
    For $q,t\in \C$ generic, the center of $aB_n(q,t)$ over $\C$ is equal to 
    \[
    C:=Z(aB_n(q,t))=\{p\in \C[x_1^{\pm},\ldots,x_n^{\pm}]^{S_n}:p(x_1,x_1^{-1},\ldots,x_n)=p(1,1,\ldots,x_n)\}.
    \]
\end{thm}
\par We will refer to symmetric Laurent polynomials satisfying the above condition as wheel polynomials.
 Parallel to the study of (super)symmetric polynomials, there is a family analogous to that of the power sum polynomials, given by 
\[
p_k^-=\sum\limits_{i=1}^n(x_i^k-x_i^{-k}), \; \quad \text{for any} \; k\in\N.
\]
The above elements are in $C$, that is, they are symmetric Laurent polynomials that satisfy the "wheel condition"
\[
p_k^-(x_1,x_1^{-1},\ldots,x_n)=p(1,1,\ldots,x_n).
\]
It is also immediate to see that the last elementary symmetric polynomial and its inverse are also in $C$, that is,
\[
e_n=x_1\cdots x_n, \quad e_n^{-1}=x_1^{-1}\cdots x_n^{-1} \in C.
\]

\subsection{Symmetric Wheel Laurent polynomials}
\par In this section, we develop a theory for wheel symmetric Laurent polynomials, reminiscent to that of (super)symmetric polynomials.
\par In the case of symmetric and supersymmetric polynomials, we have two standard sets of generators. One is given by the power sum polynomials, while the other one is a set of polynomials which are called elementary (super)symmetric polynomials. We give their definitions below, noting that the latter are usually defined as the coefficients of the formal power series of a rational generating function.
\begin{defi}
    The elementary (super)symmetric polynomials in $\C[x_1,\ldots,x_n]^{S_n} $ (respectively in $\C[x_1,\ldots,x_{2n-1},x_{2n}]^{S_n}$), denoted by ($l_k$) $e_k$, for $k\in \N$ are given as the coefficient of $T^k$ in the formal power series expansion of the following generating functions:
    \begin{align*}
      &  \text{Classical case:} \; F_n(t;x_1,\ldots,x_n)=\prod_{i=1}^n (1-x_iT)=\sum\limits_{i=0}^{\infty}e_kT^k, \\
      &  \text{Super case:} \; F_n(t;x_1,\ldots,x_n)=\frac{\prod_{i=1}^n(1+x_{2i-1}T)}{\prod_{i=1}^n( 1-x_{2i}T)}=\sum_{i=0}^{\infty}l_kT^k.
    \end{align*}
\end{defi}
\begin{rem}
    The definition of the elementary symmetric polynomials we are using is not standard here.
\end{rem}
The above rational functions have a formal power series expansion, since their constant coefficient is non-zero and equal to $1$. We also note that the set of elementary symmetric polynomials is finite while the corresponding set of elementary supersymmetric polynomials is infinite.
\begin{rem}
   The above polynomials are of great importance in the theory of (super)symmetric polynomials for two reasons.
   \begin{enumerate}
       \item They are defined by a generating function which is a rational function in $T$, making them easy to manipulate algebraically. The information of the entire family is packed in a simple rational function.
       \item They generate (super)symmetric polynomials.
   \end{enumerate} 
\end{rem}
\begin{thm}\cite{sspolys}
   The set of (super)symmetric polynomials is generated both by the set of power sum (super)symmetric polynomials and the elementary (super)symmetric polynomials. 
\end{thm}
\par Let us turn our attention back to the algebra of wheel symmetric Laurent polynomials $C$ from \cref{cenafbmw}. Notice that $C$ is exactly the multiplicative analogue of the subalgebra of symmetric polynomials $f\in \C[x_1,\ldots,x_n]^{S_n}$, satisfying the condition $f(x_1,-x_1,\ldots,x_n)=f(0,0,\ldots,x_n)$ which Nazarov in \cite{Naz} proved to be the center of the Brauer algebras $\Br_n(\delta)$ for high enough $\delta\in \Z$ as well as the type B parameters fitting in the semisimplicity parameters of $\Br_n(\delta)$.
% By type B parameters we mean even \delta.
\begin{defi} \label{elemwheel}
    Define the set of elementary symmetric wheel Laurent polynomials in $C$ (elementary wheel polynomials) to be the family $\{w_k \; | \; k\in \N\}\subseteq \C[x_1^{\pm 1},\ldots,x_n^{\pm 1}]$, where $w_n$ is given as the coefficient of $T^n$ in the formal power series expansion of the following generating function
    \begin{equation*}
       W(T)= \sum_{i=0}^{\infty}w_iT^i = \frac{\prod_{i=1}^n{(1-x_i^{-1}T)}}{\prod_{i=1}^n( 1-x_iT)}.
    \end{equation*}
\end{defi}
Since the rational functions defining these polynomials is symmetric in $x_1,\ldots,x_n$, it is clear that $w_k\in \C[x_1^{\pm 1},\ldots,x_n]^{S_n}$ for all $k\in \N$. Further, it is easy to show that they satisfy the wheel condition by making the substitution $x_2=x_1^{-1}$ in the generating function. This means that $\{w_k \; | \; k\in \N\}\subseteq C$.
%\begin{rem}
%    Note that there is a clear connection to elementary symmetric polynomials. In particular, the numerator is equal to
%    \[
%    \sum\limits_{k=0}^ne_k(x_1^{-1},\ldots,x_n^{-1})T^k,
%    \]
%    while the denominator is equal to 
%    \[
%    \sum\limits_{k=0}^n e_k(x_1,\ldots,x_n)T^k.
%    \]
%    One can view this as a strong indication that these are the analogues of the elementary symmetric polynomials for $C$.
%\end{rem}
\begin{ex}
    Let's scrutinize these polynomials for $n=2$.
    \begin{align*}
        & \frac{(1-x_1^{-1}T)(1-x_2^{-1}T)}{(1-x_1T)(1-x_2T)}=(1-x_1^{-1}T)(1-x_2^{-1}T)\sum_{k=0}^{\infty}(\sum_{a+b=k}x_1^ax_2^b)T^k= \\
        & =\sum_{n=0}^{\infty}\left((\sum_{a+b=n} x_1^ax_2^b)-(x_1^{-1}+x_2^{-1})(\sum_{a+b=n-1} x_1^ax_2^b) + x_1^{-1}x_2^{-1}(\sum_{a+b=n-2} x_1^ax_2^b) \right)T^n.
    \end{align*}
    Which yields a concrete expression for $w_k$. In particular, notice that 
    \[
    w_1=x_1-x_1^{-1}+x_2-x_2^{-1}=p_1^-
    \]
    and 
    \[
    w_2=x_1^2+x_2^2+x_1x_2-2+x_1^{-1}x_2^{-1} -x_1^{-1}x_2-x_1x_2^{-1}.
    \]
    One can calculate and see that
    \[
    w_2=\frac{1}{2}w_1^2+\frac{1}{2}p_2^-,
    \]
    so that $p_2^-=2w_2-w_1^2$.
\end{ex}
From the above example, it is plausible that the algebras generated by the $w_k$'s and the $p_k^-$'s are the same. In order to do that, we would like to connect their corresponding generating functions.
\par We calculate the generating function of the $p_k^-$'s. Consider the formal power series
\begin{align*}
   & \sum_{k=0}^{\infty}p_k^-T^k=\sum_{i=1}^n\sum_{i=0}^{\infty}(x_iT)^k-\sum_{i=1}^n\sum_{i=0}^{\infty}(x_i^{-1}T)^k= \\ 
   &=\sum_{i=1}^n \left(\frac{x_i}{1-x_iT}-\frac{x_i^{-1}}{1-x_i^{-1}T} \right) =\sum_{i=1}^n\frac{d}{dT}(-\log(1-x_iT)+\log(1-x_i^{-1}T)= \\
   & = \frac{d}{dT} \log \left( \frac{\prod_{i=1}^{n} (1-x_i^{-1}T)}{\prod_{i=1}^n(1-x_iT)}\right).
\end{align*}
\begin{lem}
    The power sum polynomials $\{p_k^- \; | \; k\in \N\}$ are given by the coefficients of $T^k$ in the formal power series expansion of the logarithmic derivative of the formal power series defining the elementary wheel polynomials $\{w_k \; | \; k\in \N\}$.
\end{lem}
Let's try to explicitly find formulas expressing the $p_k^-$'s in terms of the $w_k$'s. For that, note that $\frac{d}{dT}\log(W(T))=\frac{W'(T)}{W(T)}$. So we have to find the formal inverse of $W(T)$. Let $\frac{1}{W(T)}= \sum\limits_{i=0}^{\infty}v_iT^i$ be such that $(\sum\limits_{i=0}^{\infty} v_iT^i)(\sum\limits_{j=0}^{\infty} w_jT^j)=1$. Then we get the convolution recursive formulas 
\[
\sum\limits_{i=0}^kw_iv_{k-i}=\delta_{k,0} \quad \text{for all } \; k\geq 0.
\]
Since $w_0=1$, we can inductively prove that 
\begin{lem}
    \[
    v_k\in \Z[w_1,\ldots,w_k]
    \]
    for all $k\geq 1$.
\end{lem}
Using this, we can finally prove that the power sum wheel polynomials generate the same algebra as the elementary wheel polynomials. Note that $p_0^-=0$. Then, we have 
\begin{align*}
& \sum\limits_{k=1}^{\infty}p_k^{-}T^{k-1}=\frac{W'(T)}{W(T)}=(\sum_{i=1}^{\infty}iw_iT^{i-1})(\sum_{j=0}^{\infty} v_jT^j)=\sum_{k=1}^{\infty}\left(\sum_{j=1}^kkw_jv_{k-j} \right)T^{k-1}.
\end{align*}
which finally implies 
\begin{equation*}\label{newtonform}
    p_k^-=\sum_{j=1}^k jw_jv_{k-j}, \quad \text{for all } \; k\geq 0.
\end{equation*}
We have thus proved the following.
\begin{prop}\label{powersumelemgenthesame}
    The power sum wheel polynomials can be expressed in terms of the elementary wheel polynomials. That is, in $C$, they generate the same algebras. In formulas,
    \begin{equation*}
        \C[e_n^{\pm1}][ p_1^- ,p_2^-,\ldots]=\C[e_n^{\pm1}][ w_1,w_2,\ldots].
    \end{equation*}
\end{prop}
\section{The center of the generic Birman-Murakami-Wenzl algebras}
In this section, we determine the center of $B_n(q,t)$, for $q,t\in \C$ generic. We first connect the affine BMW algebras with the concrete ones. In particular, we define the natural surjection from $aB_n(q,t)$ to $B_n(q,t)$ and explain how this map gives us a large family of central elements in $B_n(q,t)$. Recall the definition of $aB_n(q,t)$ from \cref{afbmwdef}. The following lemma is easily deduced from the definitions of $aB_n(q,t)$ and $B_n(q,t)$.
\begin{lem}
    There is a natural surjective $\C$-algebra morphism
    \[
    \pi_n: aB_n(q,t)\rightarrow B_n(q,t)
    \]
    given by mapping $s_i$ to the generator with the same name in $B_n(q,t)$, the affine generator $x$ maps to $1$ and for $l\in\Z$, the scalars $z_0^{(l)}$ map to $t^l(\frac{t-t^{-1}}{q-q^{-1}}+1)$. The kernel of this map is equal to $(x-1)$, the ideal of $aB_n(q,t)$ generated by the element $x-1$. Further, this restricts to an algebra morphism between the centers
    \[
    \pi_n|_{Z(aB_n(q,t))}:C\rightarrow Z(B_n(q,t)).
    \]
\end{lem}
\begin{proof}
    The fact that this map is well defined and surjective is immediate. This map factors through the quotient $\tilde{\pi}_n:aB_n(q,t)/(x-1)\rightarrow B_n(q,t)$. There is also a map in the other direction, $i_n:B_n(q,t)\rightarrow aB_n(q,t)$ which is injective and such that 
    \begin{align*}
        & (\pi_n\circ i_n)\circ \tilde{\pi}_n=id_{aB_n(q,t)/(x-1)} \\
        & \tilde{\pi}_n\circ (\pi_n\circ i_n) = id_{B_n(q,t)},
    \end{align*}
    proving that $\tilde{\pi}_n$ is an isomorphism and thus the kernel of $\pi_n$ is equal to the ideal generated by $x-1$.
\end{proof}
This means we can project central elements of the affine version down to the concrete one. Let $WL[x_1,\ldots,x_n]$ denote the subalgebra of $B_n(q,t)$, the elements of which are wheel polynomials in the Jucys Murphy elements of $B_n(q,t)$, namely $x_1=t, \; x_2=s_1^2,\;\ldots\;,x_n=s_{n-1}s_{n-2}\cdots s_1^2s_2\cdots s_{n-2}s_{n-1}$. It is then clear that the image of $C$ under $\pi_n$ is equal to $WL[x_1,\ldots,x_n]$.
\begin{cor}
    \begin{equation*}
        WL[x_1,\ldots,x_n]\subseteq Z(B_n(q,t)).
    \end{equation*}
\end{cor}
\par Now we are ready to state one of the main results of this paper, as well as outline the technique which will be used in the proof.
\begin{thm}\label{centergeneric}
    For $q,t\in \C$ generic, it holds that
    \begin{equation*}
        Z(B_n(q,t))=WL[x_1,\ldots,x_n].
    \end{equation*}
Further, for all $\lambda\in \Lambda_n$, there exists an elementary wheel polynomial $w_{n(\lambda)}$ such that their evaluations in the Jucys Murphy elements provides a basis for $Z(B_n(q,t))$.
\end{thm}
\par Since $q,t\in \C$ are generic, by \cref{semibywen}, the algebras $B_n(q,t)$ are semisimple. Using the Artin-Wedderburn theorem, we get that the dimension of the center over $\C$ is equal to the number of inequivalent simple modules of $B_n(q,t)$ over $\C$. Knowing that these are in bijection with the number of partitions of $n-2f$, for $f=0,\ldots,[\frac{n}{2}]$, we deduce the following. 
\begin{lem}
    For $q,t\in \C$ generic, it holds
    \begin{equation*}
        \dim_{\C}Z(B_n(q,t))=\lvert\Lambda_n\rvert.
    \end{equation*}  
\end{lem}
\par The technique to prove \cref{centergeneric} is now clear. If we can show that the algebra $WL[x_1,\ldots,x_n]$ produces at least $\dim_\C Z(B_n(q,t))$ linearly independent elements, then we are done. In order to do that, we will show that the family $(w_n)_{n\geq 0}$ can distinguish simple modules of $B_n(q,t)$. That is, we will show that for every two simple modules $L(\lambda)$ and $L(\mu)$, there exist wheel polynomials $f,g\in C$ such that their images in $WL[x_1,\ldots,x_n]$ act by different scalars on $L(\lambda)$ and $L(\mu)$. Then a point separation argument will give us that by the above we can produce at least one linearly independent element of $WL[x_1,\ldots,x_n]$ for every simple module $L(\lambda)$ which will complete the proof.
\par Since $WL[x_1,\ldots,x_n]$ is a central subalgebra, by Schur's lemma every element $p\in WL[x_1,\ldots,x_n]$ acts on a simple module $L(\lambda)$ by some scalar. We first show that this scalar only depends on the partition $\lambda \vdash n-2f$ and not on the Murphy basis element corresponding to some path 
$T\in T^{\operatorname{ud}}_n(\lambda)$. Remember that in these Murphy bases elements $f_T$, the action of the Jucys-Murphy elements is given in \cref{JMactseminormal}.
\begin{lem}\label{independentofpath}
    Let $\lambda \vdash n-2f$ for some $f=0,\ldots,[\frac{n}{2}]$. If $f_T,f_S\in L(\lambda)$ are two elements of the Murphy basis of $L(\lambda)$, for $T,S\in T^{\operatorname{ud}}_n(\lambda)$ and $p\in WL[x_1,\ldots,x_n]$, then
    \begin{equation*}
        p(c_T(1),\ldots,c_T(n))=p(c_S(1),\ldots,c_S(n)).
    \end{equation*}
\end{lem}
\begin{proof}
    By \cref{JMactseminormal}, we have
    \begin{equation*}
        pf_T=p(c_T(1),\ldots,c_T(n))f_T
    \end{equation*}
    and 
    \begin{equation*}
        pf_S=p(c_S(1),\ldots,c_S(n))f_S.
    \end{equation*}
    Since $p$ is central, by Schur's lemma it acts by the same constant on any element of $L(\lambda)$.
\end{proof}
Using \cref{independentofpath}, for any $\lambda\vdash n-2f$, we can now choose a convenient path to $\lambda$ and calculate everything we need for the action of $WL[x_1,\ldots,x_n]$ on just the basis element corresponding to this path.
\begin{defi}\label{drunkpath}
    Let $(\lambda,f)\in \Lambda_n$. Denote by $T^{\operatorname{d}}(\lambda)=(T_0\rightarrow T_1 \rightarrow \ldots \rightarrow T_n=\lambda)$ to be the following path to $\lambda$.
      \begin{equation*}
        (\emptyset\rightarrow \Box)^{f}\rightarrow T^{\operatorname{can}}(\lambda),
    \end{equation*}
    where by $ (\emptyset\rightarrow \Box)^{f}$ we mean repeat the path $\emptyset \rightarrow \Box, \quad f$ times and $T^{\operatorname{can}}(\lambda)$ is the canonical path to $\lambda$ as defined after \cref{canpath}. Call $T^{d}(\lambda)$ the drunk path to $\lambda$ and $f$ the level of drunkenness.
\end{defi}
\begin{ex}
    Let $n=10$ and $\lambda=(2,2)\vdash 10-2\cdot3$. Identifying the drunkenness level $f=3$, we have that the drunk path to $(3,(2,2))\in \Lambda_{10}$ is 
    \ytableausetup{boxsize=0.7em}
\[
T^d(\lambda)=
\Big(
(\vcenter{\hbox{$\emptyset$}}
\rightarrow
\vcenter{\hbox{\ydiagram{1}}})^{4}
\rightarrow
\vcenter{\hbox{\ydiagram{2}}}
\rightarrow
\vcenter{\hbox{\ydiagram{2,1}}}
\rightarrow
\vcenter{\hbox{\ydiagram{2,2}}}
\Big)
\]

\end{ex}
With the above path in mind, we want to understand the multiset of $(q,t)$ contents with regards to the drunk path $T^{d}(\lambda)$. From now on, we denote $c_{T^{d}}(\lambda)(i)$ by $c_{\lambda}(i)$, for all $i$. We first fix some notation.
\begin{defi}
    For a partition $\lambda \vdash n$, let 
    \begin{itemize}
        \item The height of $\lambda$, denoted by $h(\lambda)$ to be the number of rows of $\lambda$ -1.
        \item Similarly, the width of $\lambda$, $w(\lambda)$ is the number of columns of $\lambda$ -1.
        \item The diagonal datum of $\lambda$ to be the pair $(D(\lambda),m_{\lambda})$, where 
        \[D(\lambda)=\{-h(\lambda),-h(\lambda)+1,\ldots, w(\lambda)\}\] and $m_{\lambda}:D(\lambda)\rightarrow \N$ is given by $m_{\lambda}(i)=\lvert \{\Box \in \lambda \; : \; \text{ct}(\Box)=i \}\rvert$, where $\text{ct}(\Box)=j-i$ for the box in position $(i,j)$. In other words, $m_{\lambda}(i)$ is the length of the diagonal with boxes of content $i$.
    \end{itemize}
\end{defi}
\begin{prop}\label{multisetofcontents}
    For $\lambda \vdash n-2f$, where $f=0,\ldots,[\frac{n}{2}]$, we have
    \begin{equation*}
        \{c_{\lambda}(i) \; | \; 1\leq i \leq n\}=\{t^{(f+m_\lambda(0))},(t^{-1})^{(f)},(tq^{2i})^{(m_{\lambda}(i))} \; | \; i\in D(\lambda) \setminus \{0\}\}, 
    \end{equation*}
    where the parenthesis in the exponents indicates the multiplicities of the corresponding numbers.
\end{prop}
\begin{proof}
    Remember that $T^{\text{std}}(\lambda)=(\emptyset\rightarrow \Box)^{f+1}\rightarrow T(\lambda)$. This means that, for $0\leq i\leq f$, we have $T_{2i}=\emptyset$, and $T_{2i+1}=\Box$, so that 
    \begin{equation}\label{cont1}
        \{c_{\lambda}(i) \; | \; 1\leq i \leq 2f+1=n-\lvert\lambda\rvert+1\}=\{ t,t^{-1},t,t^{-1},\ldots\}=\{t^{(f+1)},(t^{-1})^{(f)}\}.
    \end{equation}
    Now for $2f+2\leq i \leq n$, we have $T^{\operatorname{can}}(\lambda)$ as the remaining path. In every step of the path, this adds a box, giving priority to the highest rows possible. This implies that 
    \begin{equation}\label{cont2}
        \{c_{\lambda}(i) \; | \; 2f+2\leq i \leq n\} = \{(tq^{2i})^{(m_\lambda(i))} \; | \; i\in D(\lambda)\setminus \{0\} \} \cup \{ (t)^{(m_{\lambda}(0)-1)}\},
    \end{equation}
    where we separated the content zero boxes since we added the box in position $(1,1)$ in the previous step. Now combining \cref{cont1},\cref{cont2}, we get 
    \begin{equation*}
        \{c_{\lambda}(i) \; | \; 1\leq i \leq n\} = \{ (t)^{(m_{\lambda}(0)-1+f+1)},(t^{-1})^{(f)},(tq^{2i})^{(m_{\lambda}(i))} \; | \; i\in D(\lambda)\setminus \{0\} \}.
    \end{equation*}
    which completes the proof.
\end{proof}
The next step is to understand how the images of the elementary wheel polynomials $w_k$ act on every simple module $L(\lambda)$ with respect to the Murphy basis element $f_{T^{d}(\lambda)}$ corresponding to the drunk path to $\lambda$. Let 
\begin{equation}\label{X}
    W(\lambda,t)=\frac{\prod_{i=1}^n(1-(c_{\lambda}(i))^{-1}T)}{\prod_{i=1}^n(1-c_{\lambda}(i)T)}.
\end{equation}
This action encodes the action of every $w_k$ as the coefficient of $T^k$ in the formal power series of the rational function $W(\lambda,t)$ in $T$. We bring this function in a reduced form.

\begin{prop}\label{reducedexpressiongeneric}
For generic $q,t\in \C$, we have 
    \begin{equation*}
        W(\lambda,t)=\frac{\prod_{i\in D(\lambda)} (1-t^{-1}q^{-2i}T)^{m_{\lambda}(i)}}{\prod_{i\in{D(\lambda)}} (1-tq^{2i}T)^{m_{\lambda}(i)}},
    \end{equation*}
    where the expression of the rational function of the right hand side is reduced.
\end{prop}
\begin{proof}
    Using \cref{multisetofcontents}, we get 
    \begin{align*}
         & W(\lambda,t)= \frac{(1-t^{-1}T)^{m_{\lambda}(0)+f}(1-tT)^f\prod_{i\in D(\lambda)\setminus \{0\}}(1-t^{-1}q^{-2i}T)^{m_{\lambda}(i)}}{(1-tT)^{m_{\lambda}(0)+f}(1-t^{-1}T)^f\prod_{i\in D(\lambda)\setminus \{0\}}(1-tq^{2i}T)^{m_{\lambda}(i)}}= \\
         & = \frac{(1-t^{-1}T)^{m_{\lambda}(0)}\prod_{i\in D(\lambda)\setminus \{0\}}(1-t^{-1}q^{-2i}T)^{m_{\lambda}(i)}}{(1-tT)^{m_{\lambda}(0)}\prod_{i\in D(\lambda)\setminus \{0\}}(1-tq^{2i}T)^{m_{\lambda}(i)}}= \\
         & = \frac{\prod_{i\in D(\lambda)} (1-t^{-1}q^{-2i}T)^{m_{\lambda}(i)}}{\prod_{i\in{D(\lambda)}} (1-tq^{2i}T)^{m_{\lambda}(i)}}.
         \end{align*}
The expression is reduced since $t\neq \pm q^a$ for $a\in \Z$, as if we assume that for some $i,j\in D(\lambda)$, we have $tq^{2i}=t^{-1}q^{2j}$, this would imply $t=\pm q^a$, for $a\in \Z$. Further, since $q$ is not a root of unity, we have $tq^{2i}\neq tq^{2j}$, for any $i\neq j$.
\end{proof}
Now that we understand the action of elementary wheel polynomials, we are ready for the main part of the proof. For $p\in WL[x_1,\ldots,x_n]$, we simplify the notation $p(c_\lambda(1),\ldots,c_{\lambda}(n))$ to $p(\lambda)$.
\begin{prop}\label{polyseparatesimples}
    Let $(\lambda,f),(\mu,f')\in \Lambda_n$ such that $\lambda \neq \mu$. There exists $p\in WL[x_1,\ldots,x_n]$, such that $p(\lambda)\neq p(\mu)$.
\end{prop}
\begin{proof}
We prove the contrapositive. That is, assuming that for all $p\in WL[x_1,\ldots,x_n]$, it holds that $p(\lambda)=p(\mu)$, we show that $\lambda=\mu$. Since this holds for all such polynomials, it holds in particular for all the elementary wheel polynomials $w_k$ and thus $W(\lambda,t)=W(\mu,t)$. Using the reduced expressions obtained in \cref{reducedexpressiongeneric}, we get
\begin{align*}
   &\frac{\prod_{i\in D(\lambda)} (1-t^{-1}q^{-2i}T)^{m_{\lambda}(i)}}{\prod_{i\in{D(\lambda)}} (1-tq^{2i}T)^{m_{\lambda}(i)}}==\frac{\prod_{i\in D(\mu)} (1-t^{-1}q^{-2i}T)^{m_{\mu}(i)}}{\prod_{i\in{D(\mu)}} (1-tq^{2i}T)^{m_{\mu}(i)}}
\end{align*}
where both rational functions are reduced. This means that we may equate numerators and denominators. Equating the denominators yields that
\begin{align*}
   {\prod_{i\in{D(\lambda)}} (1-tq^{2i}T)^{m_{\lambda}(i)}}={\prod_{i\in{D(\mu)}} (1-tq^{2i}T)^{m_{\mu}(i)}}
\end{align*}
where both polynomials are decomposed into irreducible factors. We first show that $D(\lambda)=D(\mu)$. Without loss of generality, assume that there exists $x\in D(\lambda)\setminus D(\mu)$. This would imply that
\[
(1-tq^{2x}T)\mid  {\prod_{i\in{D(\lambda)}} (1-tq^{2i}T)^{m_{\lambda}(i)}},
\]
while
\[
(1-tq^{2x}T)\nmid {\prod_{i\in{D(\mu)}} (1-tq^{2i}T)^{m_{\mu}(i)}}
\]
yielding a contradiction. Thus $D(\lambda)=D(\mu)$. It is then immediate to see that for all $i\in D(\lambda)=D(\mu)$, we have
\begin{equation*}
    (1-tq^{2i}T)^{m_{\lambda}(i)}=(1-tq^{2i}T)^{m_{\mu}(i)},
\end{equation*}
which in turn implies that $m_{\lambda}(i)=m_{\mu}(i)$ for all $i\in D(\lambda)=D(\mu)$. This means that the partitions $\lambda,\mu$ have the same diagonals and the diagonals corresponding to $i\in D(\lambda)=D(\mu)$ have the same length. Since diagonals together with their length completely determine the shape of the partition, this implies that $\lambda=\mu$ and hence $f=f'$ as well.
\end{proof} 
The only thing left is to prove a point separation argument. Notice that the following is a direct generalization of the point separation lemma used in \cite{JungKim}. In fact, their proof goes through without any modifications.
\begin{lem}\label{separationlemma}
    Let $A$ be a $\C$ subalgebra of $\C[x_1^{\pm 1},\ldots,x_n^{\pm 1}]$. Further, assume that we have $m\in \N$ tuples
    \[
    (c_{1,1},\ldots c_{1,n}),\ldots, (c_{m,1},\ldots , c_{m,n}) \in (\C^{\times})^n.
    \]
    If for every $1\leq i\neq j\leq m$, there exists $p\in A$ such that $p(i)\neq p(j)$, where $p(i)=p(c_{i,1},\ldots,c_{i,n})$, then there exists a family $p_1,\ldots,p_m\in A$ such that
    \begin{equation}
      \det  \begin{bmatrix}
            p_1(1) &  p_1(2) &\ldots & p_1(m) \\
            p_2(1) & p_2(2) & \ldots & p_2(m) \\
            \vdots & \ldots & \ddots & \\
            p_m(1) & p_m(2) & \ldots & p_m(m)
        \end{bmatrix} \neq 0
    \end{equation}
\end{lem}
\begin{proof}
    We prove it by induction on $m$. Let $(c_1,\ldots,c_n)\in (\C^{\times})^n$. Any nonzero constant polynomial $p\in A$ is such that $p(c_1,\ldots,c_n)\neq 0$. (Non-zero constant polynomials are in $A$ since it is a $\C$ subalgebra) For inductive step, assume $m>1$ and that our assumption holds for $1\leq i \leq m-1$. Consider the first $m-1$ tuples . By our assumption, for any $1\leq i\neq j \leq m-1$ there exists $p\in A$ such that $p(i)\neq p(j)$. By the inductive hypothesis, there exists a family $p_1,\ldots,p_{m-1}\in A$, such that \begin{equation}\label{det=1}
      \det  \begin{bmatrix}
            p_1(1) &  p_1(2) &\ldots & p_1(m-1) \\
            p_2(1) & p_2(2) & \ldots & p_2(m-1) \\
            \vdots & \ldots & \ddots & \\
            p_{m-1}(1) & p_{m-1}(2) & \ldots & p_{m-1}(m-1)
        \end{bmatrix} \neq 0
    \end{equation}
    By applying elementary row operations, we can bring the above matrix to an upper uni-triangular form, thus we may assume further that 
    \begin{align*}
        & p_i(j)=0, \quad 1\leq j<i\leq m-1 \\
        & p_i(i)=1, \quad 1\leq i \leq m-1.
    \end{align*}
    In particular, we may assume that the determinant of the above matrix is equal to $1$.
    \par Assume that for every $p\in A$, 
    \begin{equation*}
      d(p):=\det  \begin{bmatrix}
            p_1(1) &  p_1(2) &\ldots & p_1(m) \\
            p_2(1) & p_2(2) & \ldots & p_2(m) \\
            \vdots & \ldots & \ddots & \\
            p(1) & p(2) & \ldots & p(m)
        \end{bmatrix} \neq 0
    \end{equation*}
    Expand by minors, taking into account \cref{det=1} to get
    \begin{equation}\label{p(n)}
        p(n)=\sum_{i=1}^{n-1}a_ip(i)
    \end{equation}
    for some $a_i\in \C$, independent of $p$. Since $A$ is a subalgebra, we have $p_{n-1}p\in A$, so that
    \begin{equation*}
        p_{n-1}p(n)=a_1p_{n-1}p(1)+\ldots+a_{n-1}p_{n-1}p(n-1)=a_{n-1}p(n-1),
    \end{equation*}
    since $p_{n-1}(i)=0$, for all $i<n-1$ and $p_{n-1}(n-1)=1$. Assume now that $p_{n-1}(n)\neq 0$. Then we may solve for $p(n)$ to get
    \[
    p(n)=\frac{a_{n-1}}{p_{n-1}(n)}p(n-1), \quad \forall p\in A.
    \]
    Choosing a non-zero constant polynomial $p\in A$, we get that $\frac{a_{n-1}}{p_{n-1}(n)}=1$, yielding that $p(n)=p(n-1)$ for every $p\in A$. This contradicts our separation assumption for $i=n-1,j=n$. Hence $p_{n-1}(n)=0$, yielding $a_{n-1}p(n-1)=0$ for every $p\in A$. Choosing $p$ non-zero constant again, gives $a_{n-1}=0$. Repeating the above argument with $p_j$ in place of $p_{n-1}$, we get that $a_j=0$, for every $j=1,\ldots,n-1$. Using \cref{p(n)}, we get
    \begin{equation*}
        p(n)=0, \quad \forall p\in A,
    \end{equation*}
    a clear contradiction. This means that there exists $p\in A$, such that $d(p)\neq 0$.
\end{proof}
We are ready to complete the proof of \cref{centergeneric}.
\begin{proof} [Proof of \cref{centergeneric}]
    Apply the separation \cref{separationlemma} to \[
    A=\{p\in \C[x_1^{\pm 1},\ldots, x_n^{\pm 1}]^{S_n} \; : \; p(x_1,x_1^{-1},\ldots,x_n)=p(1,1,\ldots,x_n)\}\] and 
    \[
    (c_{i,1},\ldots,c_{i,n})=(c_\lambda(1),\ldots,c_\lambda(n))
    \]
    for $m=\lvert \Lambda_n\rvert=\sum_{f=0}^{[\frac{n}{2}]} p(n-2f)$. This implies that there is a family of polynomials 
    \[
    \{p_{\lambda} : \lambda \in\Lambda_n\}\subseteq A
    \]
    such that the matrix 
    \[
    \big ( p_{\lambda}(\mu)\big )_{\lambda,\mu\in \Lambda_n}
    \]
    is invertible, or equivalently the rows/columns of this matrix are linearly independent. Assume that 
    \[
    P=\sum_{\lambda\in \Lambda_n} a_\lambda p_{\lambda} = 0,
    \]
    where $a_{\lambda}\in \C$. Then on every simple module $L(\mu)$, we have that $p_{\lambda}$ acts as $p_{\lambda}(\mu)$, while $P$ acts as zero. Thus 
    \[
    \sum_{\lambda\in \Lambda_n} a_{\lambda}p_{\lambda}(\mu)=0.
    \]
    But the columns of the matrix $ \big ( p_{\lambda}(\mu)\big )_{\lambda,\mu\in \Lambda_n}$ are linearly independent, yielding $a_{\lambda}=0$, for all $\lambda\in \Lambda_n$. Therefore, the set  $\{p_{\lambda} : \lambda \in \Lambda_n\}\subseteq WL[x_1,\ldots,x_n]$ is linearly independent, proving that 
    \[
    \dim_{\C}WL[x_1,\ldots,x_n]\geq \dim_\C Z(B_n(q,t)).
    \]
 Therefore, we must have $Z(B_n(q,t))=WL[x_1,\ldots,x_n]$, completing the proof.
\end{proof}
We present some immediate corollaries of \cref{centergeneric}.

    \begin{cor}\label{GZgenbyJM}
        Suppose that $q,t\in\C$ are generic. The Gelfand-Zeitlin algebra $GZ(n)=\langle Z(1),\ldots,Z(n)\rangle$ generated by all the centers $Z(i)=Z(B_i(q,t))$ is a maximal commutative subalgebra of $B_n(q,t)$ generated by the Jucys Murphy elements and their inverses $x_1=t, \; x_{i+1}=s_ix_is_i$, for  $i=1,\ldots,n-1$. 
 In formulae,
        \[
        GZ(n)=\langle x_1^{\pm 1},\ldots,x_n^{\pm 1}\rangle.
        \]
    \end{cor}
    \begin{proof}
        $GZ(n)$ is generated by elements of the form $f_1\cdots f_n$, where $f_i\in Z(i)$. Since these are given by Laurent polynomials in the JM elements, the result follows.
    \end{proof}
    In particular, we can now identify the Murphy basis of every simple $L(\lambda)$ of $B_n(q,t)$ as the GZ basis.
    \begin{cor}\label{OVlikeapproachworks}
        For any $(\lambda,f)\in \Lambda_n$, the Murphy basis of $L(\lambda)$ is (up to scalars) identified with the GZ basis of $L(\lambda)$.
    \end{cor}
    
\begin{rem}\label{remarkonOV}
    By \cref{OVlikeapproachworks}, we have an Okounkov-Vershik like approach to the representations of $B_n(q,t)$, see \cite{OV}. Since the Gelfand-Tsetlin basis elements are completely determined up to scalars by the corresponding eigenvalues, we can build up any simple module of $B_n(q,t)$ by the simple modules of $B_{n-1}(q,t)$, together with the action of the last JM element on the basis elements.
    \par We note that a similar approach to the finite dimensional representations has already appeared in \cite{IsaevOgievetsky}, but the authors have not included proofs for some of the main statements.
\end{rem}

\par To make \cref{remarkonOV} explicit, since $x_n\in C(B_n(q,t),B_{n-1}(q,t))$, we have that for any simple module $L(\lambda)$, the last JM element still acts on its restriction to $B_{n-1}(q,t)$, since
\[
x_n\in \End_{B_{n-1}(q,t)}\left(\Res^{B_n(q,t)}_{B_{n-1}(q,t)}L(\lambda)\right).
\] In particular, the summand $L(\lambda\pm\Box)$ in the decomposition of $L(\lambda)$ regarded as a $B_{n-1}(q,t)$ module is the eigenspace of $x_n$ corresponding to the $(q,t)$-content of the box added or removed from $\lambda$ to get $\lambda\pm \Box$.
\par For clarity, we update the first three levels of the branching graph with this information. As an example, the basis element of $L(\Box)$ as a simple $B_3(q,t)$-module corresponding to the path \[
T =\left( 
\vcenter{\hbox{$\emptyset$}} \rightarrow
\vcenter{\hbox{\ydiagram{1}}} \rightarrow
\vcenter{\hbox{\ydiagram{1,1}}} \rightarrow
\vcenter{\hbox{\ydiagram{1}}}\right)
\]
is completely determined by the vector $(t,tq^{-2},t^{-1}q^2)$.
    \ytableausetup{boxsize=1em}
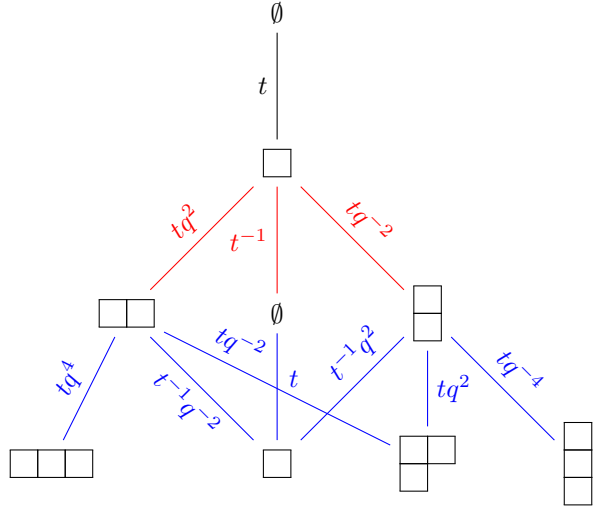
\begin{center}
\begin{tikzpicture}[yscale=2, xscale=2]
% nodes
\node (0) at (0,0) {$\emptyset$};
\node (1) at (0,-1) {$\ydiagram{1}$};
\node (11) at (1,-2) {$\ydiagram{1,1}$};
\node (111) at (2,-3) {$\ydiagram{1,1,1}$};
\node (2) at (-1,-2) {$\ydiagram{2}$};
\node(00) at (0,-2) {$\emptyset$};
\node (001) at (0,-3) {$\ydiagram{1}$};
\node (21) at (1,-3) {$\ydiagram{2,1}$};
\node (3) at (-1.5,-3) {$\ydiagram{3}$};
% edges

\draw[black] (0)--node[left]{$t$} (1);
\draw[red] (1) --node[right,above,sloped]{$tq^{-2}$} (11)
  (1)[red] --node[left]{$t^{-1}$} (00)
  (1)[red] --node[left,above,sloped]{$tq^{2}$} (2);
\draw (00)[blue] --node[right,xshift=0.1em,yshift=0.3em]{$t$} (001);
\draw[blue] (2) --node[left,below,sloped]{$t^{-1}q^{-2}$} (001)
 (2) --node[left,above,sloped]{$tq^{4}$} (3)
 (2) --node[right,above,sloped,xshift=-2em]{$tq^{-2}$} (21);
\draw[blue] (11) --node[left,above,sloped,xshift=1em]{$t^{-1}q^{2}$} (001)
 (11) --node[right]{$tq^{2}$} (21)
 (11) --node[right,above,sloped]{$tq^{-4}$} (111);
\end{tikzpicture}
\end{center}
    \begin{cor}\label{evaluationlevelitholds}
        For $q,t\in \C$ generic, it holds
        \[
        WL[x_1,\ldots,x_n]=\C[e_n^{\pm1}][w_1,w_2,\ldots]=\C[e_n^{\pm1}][p_1^{-},p_2^{-},\ldots].
        \]
    \end{cor}
    \section{Non-generic parameters.}
    In this section we work in the case when the parameters $q,t\in \C$ are chosen not to be generic. In particular, we explicitly calculate for which of the semisimplicity parameters, the center of the BMW algebras is exactly the algebra of wheel Laurent polynomials in the JM elements. We recall a theorem of Rui and Si providing necessary and sufficient conditions for the semisimplicity of $B_n(q,t)$ over $\C$. Their result holds over any field $\Bbbk$, by calculating the Gram determinants of $B_n(q,t)$.
    \begin{thm}\cite{RSi1}\label{RuiSisemisimplicity}
 Let $B_n(q,t)$ be the Birman-Murakami-Wenzl algebra over $\C$, for any $q,t\in \C$.
    \begin{enumerate}
        \item Suppose $t\notin \{q^{-1},-q\}$.
        \begin{enumerate}
            \item If $n\geq 3$, then $B_n(q,t)$ is semisimple if and only if $o(q^2)>n$ ($o(q^2)$ is the order of $q^2$) and $t\notin \cup_{k=3}^n \{\pm q^{\mp (2k-3)},\pm q^{3-k},\pm q^{k-3}\}$.
            \item $B_2(q,t)$ is semisimple if and only if $o(q^2)>2$.
            \item $B_1(q,t)$ is always semisimple.
        \end{enumerate}
        \item Assume $t\in\{q^{-1},-q\}$.
        \begin{enumerate}
            \item $B_n(q,t)$ is not semisimple, if $n$ is either even or odd with $n\geq 7$.
            \item $B_1(q,t)$ is always semisimple.
            \item $B_3(q,t)$ is semisimple if and only if $o(q^2)>3$ and $q^4+1\neq 0$.
            \item ${B_5(q,t)}$ is semisimple if and only if $o(q^2)>5, \; q^6+1\neq 0$, and $q^8+1\neq 0$.
        \end{enumerate}
    \end{enumerate}
\end{thm}
\par For the purposes of this paper, $q\in\C$ is not a root of unity since the next example shows that, when $q\in\C$ is a root of unity, then the center can be strictly larger than $WL[x_1,\ldots,x_n]$. 
\par We will have to distinguish between the following two cases.
    \begin{enumerate}
        \item $t=\pm q^{2a}$ for $a\in \Z$.  \label{even case}.
        \item $t=\pm q^{2a+1}$ for $a\in \Z$. \label{odd case}.
    \end{enumerate}
    \par We will distinguish further between the cases where $t$ is a high enough power of $q$ and when it is small enough but remains in the semisimplicity parameters of \cref{RuiSisemisimplicity}.
We present an example where the center is not $WL[x_1,\ldots,x_n]$, even when $B_n(q,t)$ is semisimple. That is, there are parameters for the BMW algebras for which the Jucys-Murphy elements do not contain all the information about the representation theory of $B_n(q,t)$, even in the semisimple range.
\par The main theorem of this section is
\begin{thm}\label{Centercompletetheorem}
    For $q\in \C$ generic and $t\in \C$ generic, or $t=\pm q^{2a}$, with $a\in \Z$, or $t=\pm q^{2a-1}$, for $\lvert a\rvert \geq n$ such that $B_n(q,t)$ is semisimple, it holds that
    \[
    Z(B_n(q,t))=WL[x_1,\ldots,x_n]
    \]
    as well as 
    \[
    GZ(n)=\langle x_1^{\pm 1},\ldots, x_n^{\pm 1}\rangle.
    \]
\end{thm}
\begin{ex}\label{smalloddpowerexample}
    Consider $n=2$ and $t=-q^{-1}$, where $q\in\C$ is not a root of unity. A parallel situation is the case where $t$ is as above and $o(q^2)=3>2$. Then $B_2(q,t)$ is commutative and by \cref{RuiSisemisimplicity} it is also semisimple. It has $3$ one dimensional simple modules labeled by the partitions $(2),\emptyset,(1,1)$ in which the unique generator $s$ acts by the scalars $q,t^{-1}=-q, -q^{-1}$, respectively (in the case $o(q^2)=3$, we have $q^3=-1$ and thus the scalars are $q,t^{-1}=-q,-q^{-1}=-q^{-3}q^2=q^2$). Then we see that $x_2=ts_1^2$ acts by the scalars $tq^2=-q,tt^{-2}=t^{-1}=-q,tq^{-2}=-q^{-3}$ (respectively $-q,-q,-1$) and as such, the action of the Jucys Murphy elements $x_1=t,\; x_2=ts_1^2$ does not distinguish between the simple modules $L((2))$ and $L(\emptyset)$. This in turn implies that the action of the central subalgebra $WL[x_1,x_2]$ on these two modules is the same, while the two simple modules are not isomorphic. In particular, for these two modules we have $W((2))=W(\emptyset)=1$ and $e_2^{\pm 1}((2))=e_2^{\pm 1}(\emptyset)=1$, where $W(\lambda,t)$ was defined in \cref{X} and $e_2^{\pm 1}=(x_1x_2)^{\pm 1}$ is the second elementary symmetric polynomial evaluated at the Jucys Murphy elements, acting on the simple modules $L((2))$ and $L(\emptyset)$. In particular, the proof of \cref{polyseparatesimples} breaks down. In this case, $\dim_\C WL[x_1,x_2]=2<3=\dim_\C B_2(q,t)=\dim_\C Z(B_2(q,t))$.
    \par One can calculate that in this case, $e\notin WL[x_1,x_2]$ and that this is one of the possible "missing" central element distinguishing between simples. The reason is that the simples $L((2))$ and $L((1,1))$ are the pullbacks of the simple modules of $H_2(q)$, hence $e$ acts on them as zero, while we easily see that it acts on $L((\emptyset))$ by the non-zero scalar $2$.
\end{ex}
\subsection{The high powers case}
\par In this section, we deal with the case of high enough powers of $q$.
Let $t=\pm q^{2a}$ or $t=\pm q^{2a-1}$ be an even power of $q$, outside of the non-semisimplicity parameters described in \cref{RuiSisemisimplicity}. The algebras $B_n(q,t)$ are then semisimple and their simple modules are still labeled by $\Lambda_n$ and in particular, \cref{JMactseminormal} holds. We want to show that the proof of \cref{centergeneric} , is true, with no modifications. For a partition $\lambda\vdash n-2f$, consider the rational function $W(\lambda,t)$. We show that the expression given in \cref{reducedexpressiongeneric} is still a reduced expression for $W(\lambda,t)$.
\begin{lem}\label{reduced expression high powers}
For $t=\pm q^{2a}$, or $t=\pm q^{2a-1}$ where $\lvert a\rvert \geq n$ and $q\in \C$ not a root of unity, it holds that
    \begin{equation*}
        W(\lambda,t)=\frac{\prod_{i\in D(\lambda)} (1-t^{-1}q^{-2i}T)^{m_{\lambda}(i)}}{\prod_{i\in{D(\lambda)}} (1-tq^{2i}T)^{m_{\lambda}(i)}},
    \end{equation*}
    where the expression of the rational function of the right hand side is reduced.
\end{lem}
\begin{proof}
    We only have to show that all factors in the numerator are different, and that they do not coincide with any of the factors in the denominator. For the first, suppose there exist $i,j\in D(\lambda)=\{-h(\lambda)+1,\ldots,w(\lambda)-1\}$ such that 
    \begin{equation*}
            tq^{2i}=tq^{2j}.
    \end{equation*}    
    Since $t$ is non-zero, this would imply $q^{2(i-j)}=1$. Since $q$ is not a root of unity, we have $i=j$. Hence, both the numerator and denominator of $W(\lambda,t)$ are reduced. We are left to prove that no factors in the numerator coincide with any factors in the denominator. Assume that $t=\pm q^{2a}$ and there exist  $i,j\in D(\lambda)$ such that 
    \begin{equation*}
            tq^{2i}=t^{-1}q^{-2j}.
    \end{equation*}
    Equivalently, since $q$ is not a root of unity we get $2a=-i-j$. For any two contents $i,j$ of a partition $\lambda \vdash n-2f$ we have that the maximum of $-i-j$ is equal to $2n-2$, for the partition $\lambda=(1^n)\vdash n$, so that $2a \leq 2n-2$ or $a\leq n-1$, a contradiction to the assumption $a\geq n$. One deals with the case $a\leq -n$ similarly, by taking the minimum of $-i-j$ for two contents to be equal to $2-2n$, so that $a\geq 1-n$ contradicting $a\leq -n$. To summarize, if $\lvert a\rvert \geq n$, then the expression of $W(\lambda,t)$ is reduced. Now assume that $t=\pm q^{2a-1}$. Then $2a-1=-i-j$. This implies that $i,j$ have different parities. Again, by examining the maximum and minimum of $-i-j$, we see that $2a-1=-i-j\leq n-1 + n-2=2n-3$ or, equivalently $a\leq n-1$ a contradiction, or $2a-1=-i-j\geq -(2n-3)$ or, equivalently $a\geq 2-n$, a contradiction to $a\leq -n$. This proves the claim.
\end{proof}
\begin{rem}
    For odd powers $t=\pm q^{2a-1}$, the above proof also goes through unchanged for $a=1-n$, but in order to make the statement symmetric, we chose the restriction $\lvert a\rvert \geq n$.
\end{rem}
By \cref{reduced expression high powers}, the proofs of \cref{polyseparatesimples},\cref{separationlemma} and \cref{centergeneric} go through with no modifications. Thus, we can state
\begin{thm}\label{centerhighpowersnonrootofunity}
    If $q\in\C$ is not a root of unity and $t=\pm q^{2a}$ or $t=\pm q^{2a-1}$ for $\lvert a\rvert\geq n$, then
    \begin{equation*}
        Z(B_n(q,t))=WL[x_1,\ldots,x_n].
    \end{equation*}
\end{thm}
\begin{cor}
    For $q,t\in\C$ as in \cref{centerhighpowersnonrootofunity}, it holds that 
    \begin{equation*}
        GZ(n)=\langle x_1^{\pm 1},\ldots,x_n^{\pm 1}\rangle.
    \end{equation*}
\end{cor}
\begin{ex}
Let $n=2$ and $t=1=q^0$. Then $B_2(q,1)$ is still semisimple with $3$ one dimensional simple modules. We still show that even if the expressions $W(\lambda,t)$ are not reduced, they still completely determine partitions. In this case, $t=t^{-1}$, so that
\begin{align*}
          &  W(\emptyset)=1, \quad
           W((2))=\frac{1-q^{-2}T}{1-q^2T},  \quad
            W((1,1))=\frac{1-q^2T}{1-q^{-2}T}.
\end{align*}
\end{ex}

\begin{ex}\label{examplewithpairingsfirst}
To show that there is nothing special to the above example, pick $n=4$ and $t=q^2$. By \cref{RuiSisemisimplicity} we have $B_4(q,q^2)$ is semisimple. Note that in this case, in order to simplify $W(\lambda,t)$ for $\lambda\vdash4,2,0$, we need to find pairs of $(q,t)$-contents such that they multiply to $1$, equivalently $q^2q^{2i}=q^{-2}q^{-2j}$. These pairings for $i,j\in\{-3,\ldots,3\}$ occur for $i=0,j=-2$, as well as for $i=j=-1, \quad i=-3, \; j=1, \quad i=1, \;j=-3$. In particular, we notice that in order for a partition to see a pairing, it needs to have a box in the $-1$ diagonal. If it does not, then no pairings occur, hence the expression of $W(\lambda,t)$ remains reduced. Now for the partitions that see the diagonal $-1$, they can have at most one box in the $-1$ diagonal, this box is paired with itself and hence the terms $\frac{1-t^{-1}q^{2}T}{1-tq^{-2}T}$ are deleted from $W(\lambda,t)$. It is also clear that the pair $(-3,1)$ will never be seen by a partition of size less than or equal to $4$. 
\par To summarize, $W(\lambda,t)$ remains reduced for partitions $\lambda$ not having a box in the $-1$ diagonal, and for boxes that see the $-1$ diagonal, we have $m_{\lambda}(-1)=1$ and need to delete the term corresponding to $tq^{-2}$ in $W(\lambda,t)$, as well as the pairs $(-2,0)$ or $(-3,-1)$ if they are seen by $\lambda$. To make this a combinatorial calculation with Young diagrams, we color diagonals whose contents multiply to 1 with the same color and then delete them from the expression. The shape that is left with no possible pairings yields the reduced expression of $W(\lambda,t)$. Below we give the reduced forms of $W(\lambda,t)$ for every $\lambda\vdash4,2,0$ and illustrate how to get the reduced form from the Young diagrammatic calculation for $\lambda=(1^4)$.
        \begin{align*}
          &  W((\emptyset))=1 \\
          & W((2))=\frac{(1-q^{-2}T)(1-q^{-4}T)}{(1-q^2T)(1-q^4T)}, \quad W((1,1))=\frac{(1-q^{-2}T)}{(1-q^2T)}, \\
          & W((4)) = \frac{\prod_{i=0}^3 (1-q^{-2(i+1)}T)}{\prod_{i=0}^3(1-q^{2(i+1)}T)}, \quad W((1^4)) =\frac{(1-q^{4}T)}{(1-q^{-4}T)}, \\
          & W((2,1,1))=\frac{(1-q^{-4}T)}{(1-q^4T)}, \quad W((2,2))=\frac{(1-q^{-2}T)^2(1-q^{-4}T)}{(1-q^{2}T)^2(1-q{4}T)},  \\
          & W((3,1))=\frac{(1-q^{-2}T)(1-q^{-4}T)(1-q^{-6}T)}{(1-q^{2}T)(1-q^{4}T)(1-q^{6}T)},
        \end{align*}
        with the pictorial calculation for $W((1^4))$
        \ytableausetup{boxsize=1.3em}
        \[\vcenter{\hbox{\scalebox{0.7}{
        \begin{ytableau}
            *(blue) 0 \\
            *(red) -1 \\
            *(blue) -2 \\
            -3
        \end{ytableau}}}}
        =
        \vcenter{\hbox{\scalebox{0.7}{\begin{ytableau}
            *(blue) 0 \\
              \\
            *(blue) -2 \\
            -3
        \end{ytableau}}}}
        =
         \vcenter{\hbox{\scalebox{0.7}{\begin{ytableau}
            \\
            \\
            \\
            -3
        \end{ytableau}}}}
        \]
        which means that the only term surviving in $W((1^4))$ is the one corresponding to $tq^{-6}$, with multiplicity $m_{(1^4)}(-3)=1$. It is clear that these expressions separate simple modules, hence
         \[
        Z(B_4(q,q^2))=WL[x_1,x_2,x_3,x_4].
        \]
    \end{ex}
    \par To complete the $B_4(q,t)$ case, we encourage the reader to use the procedure described in \cref{examplewithpairingsfirst} to show that
        \[
        Z(B_4(q,q^4))=WL[x_1,x_2,x_3,x_4].
        \]
    \subsection{Small even powers of $q$}\label{smallevenpowers}
    In this subsection, we deal with the case $t=\pm q^{2a}$, where $a\in \{1-n,\ldots,n-1\}$. Notice that this is the set of all possible contents of partitions $\lambda \vdash n-2f$, for $f=0,\ldots,[\frac{n}{2}]$.
    \par Towards that, we will actually solve the case $a\in\{0,\ldots,n-1\}$, as the proofs for negative powers are just reflections of the combinatorics we use in this section. That is, one can replace $\lambda$ by its conjugate partition $\lambda^t$ to get the results for negative powers. Hence, assume $0\leq a\leq n$ and $n\geq 3$, since we dealt with the case $n=2$ in the previous section. In this case, in order for $B_n(q,q^{2a})$ to be semisimple, we need $2a>n-3$ or, equivalently, $2a\geq n-2$. We thus need to deal with the cases $a\in\{[\frac{n}{2}]-1,\ldots,n-1\}$.
    \begin{defi}
        For a partition $\lambda\vdash n-2f$ and $i,j\in D(\lambda)$, we say that the diagonals $i,j$ are $(q,t)$-paired, if $c_{\lambda}(i)c_{\lambda}(j)=1$. Equivalently, this means $tq^{2i}=(tq^{2j})^{-1}$.
    \end{defi}
   \par  For any $f=0,\ldots,[\frac{n}{2}]$ and $\lambda\vdash n-2f$, its Young diagram fits inside the Young diagram of the partition $(n,[\frac{n}{2}],[\frac{n}{3}],\ldots,[\frac{n}{n}])$ shown below.
    \[\scalebox{0.7}{\ytableausetup{boxsize=2em}
    \begin{ytableau}
        0 & 1 & \ldots & \ldots & \ldots & \scriptstyle{n-1} \\
        -1 & \ldots & \ldots & \scriptscriptstyle{[\frac{n}{2}]-1} \\
        \vdots & \ldots \\
        -a & ? \\
        \scriptscriptstyle{-a-1} \\
        \vdots \\
       \scriptscriptstyle{1-n}
    \end{ytableau}}\]
    Since $a\geq[\frac{n}{2}]-1$, we have that the $-a$ diagonal will be met at the $[\frac{n}{[\frac{n}{2}]-1}]=2$ part of the above diagram, or lower. The question mark indicates that if $a=[\frac{n}{2}]-1$, then there is exactly one box that can fit in a partition $\lambda\vdash n-2f$ next to it, but we only have $1$ box in every diagonal after it. 
    \par In particular, we notice the following. The box with content $-a$ will appear at most once in every partition $\lambda\vdash n-2f$. Moreover, for a chosen $a$, we see that the diagonals that are $(q,t)$-paired are $(-a,-a)$, and all pairs of the form $(-a-i,-a+i)$, for $i=1,\ldots,n-a-1$. Note that even though this big partition sees pairs occurring outside the first column, a partition $\lambda\vdash n-2f$ will only see pairs in the first column. That is because if a partition is able to see the $-a$ diagonal, then in order for it to see the pair $(-2a-1,1)$, it would need to be a partition of at least $2a+3\geq n+1$, which is absurd.
 \par   Hence, for a partition $\lambda\vdash n-2f$, if $-a\notin D(\lambda)$, then there are no pairs in $\lambda$. For $-a\in D(\lambda)$, the pairs seen by $\lambda$ are $(-a-i,-a+i)$, for $i=1,\ldots,\min\{h(\lambda)-a-1,a\}$. Moreover, notice that if $\lambda$ can see the pair $(-a-i,-a+i)$, then $m_{\lambda}(-a-i)=1=m_{\lambda}(-a+i)$. That is, because if you were able to see $-a-i$ more than once, that would mean you are a partition of at least $2(a+i+1)+2=2a+4+2i \geq n+2(i+1)$, respectively if you are able to see both $-a+i$ and $-a-i$, that means you are a partition of at least $2(a-i+1)+2+(a+i)-(a-i)=2a+4\geq n+1$. 
 \begin{defi}
     For $\lambda\vdash n-2f$, we denote by $P(\lambda)$ to be the set containing all $i\in D(\lambda)$ that are $(q,t)$-paired.
 \end{defi}
 \begin{ex}\label{examplewithqtpairs}
     For $n=10, \; a=4$ and $\lambda=(2,1^8)$, we have 
     \[P(\lambda)=\{-4,-5,-3,-6,-2,-7,-1\}.\]
     To illustrate pairings, we draw the Young diagram of the partition, coloring the paired diagonals with the same color. Uncolored diagonals do not have a $(q,t)$-pair and if there is a single occurence of a color, this means the diagonal is $(q,t)$-paired with itself.
     
     \[
     \ytableausetup{boxsize=1.3em}\scalebox{0.7}
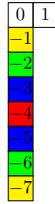
{
     \begin{ytableau}
         0 & 1 \\
        *(yellow) -1 \\
        *(green) -2 \\
        *(blue) -3 \\
        *(red) -4 \\
        *(blue) -5 \\
        *(green) -6 \\
        *(yellow) -7
     \end{ytableau}}
     \]
 \end{ex}
 It follows from the discussion above that
 \begin{lem}\label{redexpressioneven}
     For $a\geq n-2$ and $\lambda \vdash n-2f$, the reduced expression of $W(\lambda,t)$ is
     \begin{equation*}
         W(\lambda,t)=\prod_{i\notin P(\lambda)} \frac{(1-t^{-1}q^{-2i}T)^{m_{\lambda}(i)}}{(1-tq^{2i}T)^{m_{\lambda}(i)}}.
     \end{equation*}
 \end{lem}
 \begin{ex}
     For \cref{examplewithqtpairs}, by deleting all colored pairs in the Young diagram, we get
     \[
     W((2,1^8))=\frac{(1-t^{-1}T)(1-t^{-1}q^{-2}T)}{(1-tT)(1-tq^2T)}.
     \]
 \end{ex}
 We are now in a position to prove that
 \begin{lem}
     For $\lambda,\mu\in \Lambda_n$, it holds that
     \[
     W(\lambda,t)=W(\mu,t) \iff \lambda=\mu.
     \]
 \end{lem}
 \begin{proof}
    Let $W(\lambda,t)=W(\mu,t)$.
     If $-a\notin D(\lambda)$ or $-a\notin D(\mu)$, then the expression of $W(\lambda,t)$ or $W(\mu,t)$ completely determines the partition, so there is nothing to prove. Hence, suppose $-a\in D(\lambda)\cap D(\mu)$. Since the expression in \cref{redexpressioneven} is reduced, we get that 
     \begin{equation*}\label{justanequality}
         \prod_{i\notin P(\lambda)} (1-tq^{2i}T)^{m_{\lambda}(i)}=\prod_{i\notin P(\mu)} (1-tq^{2i}T)^{m_{\mu}(i)}
     \end{equation*}     
     where both are reduced. 
      Suppose, without loss of generality, that $(-a-i,-a+i)$ is a $(q,t)$-pair seen by $\lambda$ but not $\mu$. Then $(1-tq^{-a+i}T)$ is a factor of the right hand side but not of the left hand side, a contradiction. This shows $P(D(\lambda))=P(D(\mu))$. Further, by \cref{justanequality}, we get $D(\lambda)\setminus P(D(\lambda))=D(\mu)\setminus P(D(\lambda))$ and $m_{\lambda}(i)=m_{\mu}(i)$ for all such $i$. But then, $D(\lambda)=D(\mu)$ and $m_{\lambda}(i)=m_{\mu}(i)$ for all $i\notin P(D(\lambda))=P(D(\mu))$ and for $i\in P(D(\lambda))=P(D(\mu))$, it holds that $m_{\lambda}(i)=1=m_{\mu}(i)$. This completes the proof. 
 \end{proof}

\begin{cor}
    For $q\in \C$ not a root of unity and $t=\pm q^{2a}$ where $a\in \Z$ is such that $B_n(q,t)$ is semisimple, then 
    \[
    Z(B_n(q,t))=WL[x_1,\ldots,x_n].
    \]
    Moreover,
    \[
    GZ(n)=\langle x_1^{\pm 1},\ldots, x_n^{\pm 1}\rangle.
    \]
\end{cor}
\begin{rem}
    For $t=-q^{2a}$, all the above arguments go through unchanged. When $t=q^{2a}$ with $a< 3-n$, all the above arguments are valid by replacing $\lambda$ with $\lambda^t$.
\end{rem}
\section{Small odd powers.}
In this section, we show that with the exception of $B_3(q,q^{\pm 1})$ and $B_3(q,-q)$, where $q\in\C$ is such that $o(q)=\infty$, the algebra $WL[x_1,\ldots,x_n]$ does not separate all simple modules of $B_n(q,t)$ when $t=q^{2a-1}$ and $1\leq a\leq n-1$. Hence, for small odd powers, we demonstrate that $\dim_{\C}WL[x_1,\ldots,x_n]<\dim_{\C}Z(B_n(q,t))$, where $t=q^{2a-1}$ and $1\leq a \leq n-1$.
\par One can easily calculate for $n=3$ and $t=q,-q,q^{-1}$ and see that $WL[x_1,x_2,x_3]$ separates simple modules as before. Similarly, for $n=2$ and $a=1$ one calculates $W((1^2))=W(\emptyset)$ so that the center of $B_2(q,q)$ is not equal to $WL[x_1,x_2]$.
\par Now suppose $t=q^{2a-1}$, with $1\leq a\leq n-1$ and $n\geq 3$.
\begin{lem}
    For $t=q^{2a-1}$, with $1\leq a\leq n-1$ and $n\geq3$ the following hold.
    \begin{enumerate}
        \item If $(n,a)=(3,1)$, then $Z(B_3(q,q))=WL[x_1,x_2,x_3]$.
        \item If $(n,a)\neq (3,1)$, then there exist partitions $\lambda,\mu$ with $\lambda\neq \mu$ and $W(\lambda,t)=W(\mu,t)$. In other words, $\dim_{\C}WL[x_1,\ldots,x_n]<\dim_{\C}Z(B_n(q,q^{2a-1})$.
    \end{enumerate}
\end{lem}
\begin{proof}
    Let $n=3,a=2$. Then it is easy to see that $W((1^3))=W((1))$. Now suppose $n\geq 4$ and $a=1$. Then \[
    W((n-2,2))=W((n-2)).\] If $a\geq 2$, then one observes that \[W((n-a,1^a))=W((n-a,1^{a-2})).\]
\end{proof}
One can also calculate and see that for $t=q^{-1},-q$, it holds that $Z(B_3(q,t))=WL[x_1,x_2,x_3]$, but that is not the case for $B_5(q,t)$.
\begin{rem}
    We have thus covered all the cases for which the center is equal to the algebra of wheel Laurent polynomials evaluated at the Jucys-Murphy elements, when $q$ is not a root of unity.
\end{rem}
\section{An equivalent condition for blocks in even case}
In this section, we explore whether in the non-semisimple cases, the central subalgebra $WL[x_1,\ldots,x_n]$ is large enough to separate the blocks of $B_n(q,t)$. In particular, we will show that this is the case for $t=\pm q^{2a}$, when $a\in \Z$ and $q$ generic.
\subsection{Blocks of Birman-Murakami-Wenzl algebras}
 Rui and Si in \cite{RSi2} found combinatorial conditions describing when two cell modules $\Delta(\lambda)$ and $\Delta(\mu)$ are in the same block, for $\lambda,\mu\in \Lambda_n$. We recall the necessary combinatorial notions needed to state their result. 
\begin{defi}\cite{RSi2}\label{defofblocks}
    Let $\lambda\vdash n$ and $\mu\vdash n-2f$. We say that $\lambda$ is $(f,\mu)$-admissible over $\C$, if
    \begin{enumerate}
        \item $\mu\subset \lambda$.
        \item There is a pairing of all diagonals in $\lambda/\mu$. In particular, for every $i\in D(\lambda/\mu)$, there exists $j\in D(\lambda/\mu)$, such that $c_{\lambda/\mu}(i)c_{\lambda/\mu}(j)=1$ and $m_{\lambda/\mu}(i)=m_{\lambda/\mu}(j)$ for any such pair.
        \item If $i\in D(\lambda/\mu)$ is such that $c_{\lambda/\mu}(i)=q$ and diagonals $(i,i-1)$ are $(q,t)$-paired in $\lambda/\mu$, then $m_{\lambda/\mu}(i)=m_{\lambda/\mu}(i-1)$ is even. In Young diagrams, whenever we encounter 
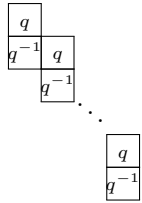
\[
\ytableausetup{boxsize=1.2em}
\begin{ytableau}
  {\scriptstyle q}      & \none \\
  {\scriptstyle q^{-1}} & {\scriptstyle q} \\
  \none                 & {\scriptstyle q^{-1}} \\
  \none                 & \none & \none[\ddots] \\
  \none                 & \none & \none & {\scriptstyle{q}} \\
  \none                 & \none & \none & {\scriptstyle{q^{-1}}} 
\end{ytableau}
\]
in $\lambda/\mu$, then the number of columns is even.

        \item If $i\in D(\lambda/\mu)$ is such that $c_{\lambda/\mu}(i)=-q^{-1}$ and the diagonals $(i,i+1)$ are $(q,t)$-paired in $\lambda/\mu$, then $m_{\lambda/\mu}(i)=m_{\lambda/\mu}(i+1)$ is even. Similarly in Young diagrams, if we encounter 
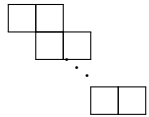
\[
\ytableausetup{boxsize=1em}
\begin{ytableau}
  {}     &  {}  \\
  \none & {} & {} \\
  \none                 & \none & \none[\ddots] \\
  \none                 & \none & \none & {} & {} \\
\end{ytableau}
\]
in $\lambda/\mu$, where the $(q,t)$-content of the upper left box is $-q^{-1}$, then the number of rows in this configuration is even.
    \end{enumerate}
\end{defi}
\begin{ex}
    Let $t=q, \; n=8$ and $\lambda=(4,2,2), \; \mu=(4,1,1)$. Clearly condition (1) is satisfied. Now for condition (2), we draw the Young diagram of $\lambda$ with its $(q,t)$-contents in every diagonal to see that there are two diagonals in $\lambda/\mu$, in the situation of condition (3), since the diagonal $0$ has $(q,t)$-content equal to $q$ and is paired to $-1$.
    \[\ytableausetup{boxsize=1.5em}\scalebox{0.7}{
    \begin{ytableau}\label{fig}
        q & q^{3} & {q^5} & q^{7} \\
        q^{-1} & q \\
        q^{-3} & q^{-1}
    \end{ytableau}}
    \]
    Then $\lambda/\mu$ is of shape
    \[\scalebox{0.7}{
    \begin{ytableau}
        q \\
        q^{-1}
    \end{ytableau}}
    \]
    In order for $\lambda$ to be $(1,\mu)$ admissible, we would need $m_{\lambda/\mu}(0)=m_{\lambda/\mu}(-1)$ to be even, which is not the case, proving that $(4,2,2)$ is not $(1,(4,1,1))$ admissible. 
\end{ex}
\begin{ex}
    Let us also present an example of an admissible pair of partitions. Set $t=q^2, \; n=8$ and $\lambda=(4,2,2), \; \mu=(4)$. Again $\mu\subset \lambda$ and the skew shape $\lambda/\mu$ is such that diagonals $(-1,-1)$ and $(-2,0)$ are paired.
    \[\scalebox{0.7}{
    \begin{ytableau}
        *(red) -1 & *(blue) 0 \\
        *(blue) -2 & *(red) -1
    \end{ytableau}}
    \]
    So condition (2) is satisfied. Further, since no diagonal has $(q,t)$-content equal to $q$ or $-q^{-1}$, there is nothing to check for conditions (3),(4). Hence, $\lambda$ is $(2,\mu)$ admissible.
\end{ex}
\begin{defi}\cite{RSi2}\label{admissibility condition}
    For $(\lambda,f),(\mu,l)\in \Lambda_n$, we say that $(\lambda,f)$ and $(\mu,l)$ are block equivalent and write $(\lambda,f)\overset{b}\sim(\mu,l)$, if $\lambda$ is $(l_1,\lambda \cap \mu)$-admissible and $\mu$ is $(l_2,\lambda\cap\mu)$-admissible, where $2l_1=\lvert \lambda \rvert -\lvert \lambda \cap \mu \rvert, \; 2l_2=\lvert \mu \rvert -\lvert \lambda \cap \mu \rvert$ and $\lambda\cap \mu$ is the partition occuring from the intersection of the two Young diagrams. In formulas, $(\lambda\cap \mu)_i=\min\{\lambda_i,\mu_i\}$.
\end{defi}
\par We state the result of Rui and Si for blocks of Birman-Murakami-Wenzl algebras over $\C$. For details on blocks of cellular algebras, see \cite{GL}. Note that their result holds more generally over any field $\Bbbk$ of characteristic not equal to $2$.
\begin{thm}\cite{RSi2}\label{theorem for blocks}
    Suppose $q,t\in \C$ and $o(q^2)>n$. Then two cell modules $\Delta(\lambda,f),\Delta(\mu,l)$ are in the same block, if and only if $(\lambda,f)\overset{b}\sim(\mu,l)$.
\end{thm}
\subsection{Separation of blocks}
By \cref{theorem for blocks}, the block of the BMW algebras $B_n(q,t)$ corresponding to $(\lambda,f)\in \Lambda_n$ consists of the partitions $(\mu,l)$ that are admissible, in the sense of \cref{admissibility condition}. In this section, we show that one can have an alternative description of these blocks, by using the action of the commutative subalgebra $WL[x_1,\ldots,x_n]$ of $B_n(q,t)$.
\par The blocks of $B_n(q,t)$ give a set partition of the set of isomorphism classes of simple modules, $\Lambda_n$. Now for a central element $c\in Z(B_n(q,t))$, using Schur's lemma we get that $c$ acts on any simple module as a scalar. Denote this scalar by $\chi_{\lambda}(c)$. This lifts to an algebra morphism, called the central character corresponding to $\lambda$, 
\begin{align*}
    \chi_{\lambda}:Z(B_n(q,t))\rightarrow \C.
\end{align*}
It is a well known result that central characters separate blocks, in the sense that for partitions $\lambda,\mu$, it holds that $\chi_{\lambda}=\chi_{\mu}$ if and only if $\lambda\overset{b}\sim\mu$.
\par Even though in the non-semisimple cases we cannot use the previous techniques to show that the center is equal to $WL[x_1,\ldots,x_n]$ as the Artin-Wedderburn theorem is not available in the non-semisimple case, we will show that this is a large enough central subalgebra to separate blocks. Comparing with other cases such as the partition algebra, \cite{CRE}, we make the following conjecture.
\begin{conj}
 For $q\in\C$ generic and $t=q^{2a}$ such that $B_n(q,t)$ is non-semisimple, it holds that \[Z(B_n(q,q^{2a})=WL[x_1,\ldots,x_n].\]
 \end{conj}
 \begin{rem}
     Note that progress has been made to compute centers in the non-semisimple cases. In particular, the center of the  non-semisimple walled Brauer algebras $B_{r,1}(\delta)$ was computed in \cite{chavli2024centerwalledbraueralgebra}.
 \end{rem}
\begin{thm}\label{blockseparation}
    Let $q\in \C$ be generic and $t=\pm q^{2a}$, for $a\in \{0,\ldots,\lfloor\frac{n}{2}\rfloor-2\}$, such that $B_n(q,t)$ is not semisimple. Then for two cell modules $\Delta(\lambda,f),\Delta(\mu,l)$ we have $(\lambda,f)\overset{b}\sim(\mu,l)$, if and only if $W(\lambda,t)=W(\mu,t)$.
\end{thm}
\begin{rem}
    The description of blocks for the BMW algebras when $t=\pm q^{2a}$ is simpler than that for odd powers of $q$, since conditions (c),(d) will never appear in the former case. Note also that in the case $t=\pm q^{2a+1}$, \cref{blockseparation} does not hold.
\end{rem}
\begin{ex}
    Let $n=2, \; t=q^{-1}$. Then $B_2(q,t)$ is not semisimple. In particular, for the partition $(2)$, we have a pairing of $(0,1)$ as their corresponding $(q,t)$-contents are $q^{-1},q$. It is easy to check that $\emptyset \overset{b}\nsim(2)$, since condition (d) of \cref{defofblocks} but $W(\emptyset)=W((2))=1$. 
\end{ex}
\par We first have to show that even in the non-semisimple case, the central subalgebra $WL[x_1,\ldots,x_n]$ acts on the simple heads of the cell modules by the same constant.
\par In \cite{Xi}, Xi classified the simple modules of $B_n(q,t)$, for any $t\in \C$. In particular, if $t\notin\{-q,q^{-1}\}$, or $t\in\{-q,q^{-1}\}$ and $n$ is odd, the simple modules are indexed by $\Lambda_n$. When $t\in \{-q,q^{-1}\}$ and $n$ is even, then the simple modules are indexed by $\Lambda_n\setminus\{(\frac{n}{2},\emptyset)\}$. That is, 
\[
L((\frac{n}{2},\emptyset))=\Delta((\frac{n}{2},\emptyset))/\operatorname{rad}(\Delta((\frac{n}{2},\emptyset))=\{0\}.
\] Since for this entire section $t=q^{2a}$ and $o(q^2)>n$, the simple modules are never trivialized.
\par For $\lambda\vdash n-2f$, the path $T^{d}(\lambda)$ is maximal with respect to both the partial order on $T^{\operatorname{ud}}(\lambda)$ defined by Enyang in \cite{EnyangMurphyBasisold} and the partial order defined by Rui and Si in \cite{RSi2}. This together with the definition of a Murphy basis of a cell module $\Delta(\lambda)$ implies that the basis element $m_T$ is a simultaneous eigenvector of $x_1,\ldots,x_n$, with the same eigenvalues as before. Since $L(\lambda)=\Delta(\lambda)/\operatorname{rad}(\Delta(\lambda))$, we get 
\begin{cor}
    For the simple $L(\lambda,f)$ and the Murphy basis element $m_T$ of $\Delta(\lambda,f)$ corresponding to the path $T^{d}(\lambda)$, it holds that
    \[
    x_i (m_T+ \operatorname{rad}(\Delta(\lambda,f))) = c_T(i)(m_T+\operatorname{rad}(\Delta(\lambda,f)))
    \]
    for all $i=1,\ldots,n$.
\end{cor}
This means that $WL[x_1,\ldots,x_n]$ acts on any simple module $L(\lambda)$ by the same scalar as in the semisimple case. We first show that $W(\lambda,t)$ is compatible with inclusions $\mu \subset \lambda$.

\begin{defi}
 For $\mu\subseteq \lambda$, we let 
        \[
        W(\lambda/\mu,t)=\frac{\prod\limits_{i\in D(\lambda/\mu)}(1-t^{-1}q^{-2i}T)^{m_{\lambda/\mu}(i)}}{\prod\limits_{i\in D(\lambda/\mu)}(1-tq^{2i}T)^{m_{\lambda/\mu}(i)}}
        \]
    \end{defi}
    
 \begin{lem}\label{7.11}
        For $\mu\subseteq \lambda$, it holds
        \begin{align*}
           W(\lambda,t)=W(\mu,t)W(\lambda/\mu,t). 
        \end{align*}
    \end{lem}
    \begin{proof}
    For any $i\in D(\lambda)$, we have $m_{\lambda}(i)=m_{\mu}(i)+m_{\lambda/\mu}(i)$. Now calculate
        \begin{align*}
           & W(\lambda,t)=\prod_{i\in D(\lambda)}\frac{(1-t^{-1}q^{-2i}T)^{m_{\mu}(i)}(1-t^{-1}q^{-2i}T)^{m_{\lambda/\mu}(i)}}{(1-tq^{2i}T)^{m_{\mu}(i)}(1-tq^{2i}T)^{m_{\lambda/\mu}(i)}}= \\
           & = \prod_{i\in D(\mu)}\frac{(1-t^{-1}q^{-2i}T)^{m_{\mu}(i)}}{(1-tq^{2i}T)^{m_{\mu}(i)}}\prod_{i\in D(\lambda/\mu)}\frac{(1-t^{-1}q^{-2i}T)^{m_{\lambda/\mu}(i)}}{(1-tq^{2i}T)^{m_{\lambda/\mu}(i)}}=W(\mu,t)W(\lambda/\mu,t),
        \end{align*}
        where the first product (respectively, second) runs over $i\in D(\mu)$ (over $D(\lambda/\mu) )$, since if $i\notin D(\mu)$ ($i\notin D(\lambda/\mu)$) then $m_{\mu}(i)=0$ (respectively, $m_{\lambda/\mu}(i)=0$).
    \end{proof}
    \begin{lem}\label{7.12}
        For $\mu \subseteq \lambda$, we have
        \[
        W(\lambda,t)=W(\mu,t),
        \]
        if and only if all $i\in D(\lambda/\mu)$ are $(q,t)$-paired. Further, for all such pairs $i,j\in D(\lambda/\mu)$, we have $m_{\lambda/\mu}(i)=m_{\lambda/\mu}(j)$.
    \end{lem}
    \begin{proof}
        We have $W(\lambda,t)=W(\mu,t)$ if and only if $W(\lambda/\mu,t)=1$. The latter occurs if and only if 
        \begin{align}\label{8.14}
          \frac{\prod\limits_{i\in D(\lambda/\mu)}(1-t^{-1}q^{-2i}T)^{m_{\lambda/\mu}(i)}}{\prod\limits_{i\in D(\lambda/\mu)}(1-tq^{2i}T)^{m_{\lambda/\mu}(i)}}=1.
        \end{align}
        Since both the numerator and the denominator are irreducible in the case $q$ is not a root of unity, \cref{8.14} happens if and only if for all $i\in D(\lambda/\mu)$ there exists $j\in D(\lambda/\mu)$ such that $i,j$ are $(q,t)$-paired and $m_{\lambda/\mu}(i)=m_{\lambda/\mu}(j)$.
    \end{proof}
    \begin{lem}
        For $(\mu,l),(\lambda,f)\in \Lambda_n$
        \[
        W(\lambda,t)=W(\mu,t)
        \]
        if and only if all $i\in D(\lambda/\lambda\cap\mu)$ and $i\in D(\mu/\lambda\cap\mu)$ are $(q,t)$-paired.
    \end{lem}
    \begin{proof}
        If all such elements are $(q,t)$-paired, then $W(\lambda/\lambda\cap\mu,t)=1=W(\mu/\lambda\cap\mu,t)$, so that
        \[
        W(\lambda,t)=W(\lambda\cap\mu,t)=W(\mu,t).
        \]
        Conversely, suppose $W(\lambda,t)=W(\mu,t)$. Then by decomposing according to \cref{7.11}, we have that $W(\lambda/\lambda\cap\mu,t)=W(\mu/\lambda\cap\mu,t)$. We need to show that both are equal to $1$. Towards that, assume that $W(\lambda/\lambda\cap\mu,t)\neq 1$. This implies the existence of $i\in D(\lambda/\lambda\cap\mu)$ such that $m_{\lambda/\lambda\cap\mu}(i)\neq 0$. By simplifying $W(\lambda/\lambda\cap\mu,t)$ to get a reduced expression by separating terms that are paired or not, we get that the reduced expression for $W(\lambda/\lambda\cap\mu,t)$ is equal to
        \[
        W(\lambda/\lambda\cap\mu,t)=A\cdot B,
        \]
        where 
        \begin{equation*}
            A=\prod\limits_{i \notin P(\lambda/\lambda\cap\mu)}\frac{(1-t^{-1}q^{-2i}T)^{m_{\lambda/\lambda\cap\mu}(i)}}{(1-tq^{2i}T)^{m_{\lambda/\lambda\cap\mu}(i)}}
        \end{equation*}
        and 
        \begin{equation*}
            B=\prod\limits_{\substack{(i,j)\in D(\lambda/\lambda\cap\mu),\\{c_{\lambda/\lambda\cap\mu}(i)c_{\lambda/\lambda\cap\mu}(j)=1,} \\m_{\lambda/\lambda\cap\mu}(j)<m_{\lambda/\lambda\cap\mu}(i)}}\frac{(1-t^{-1}q^{-2i}T)^{m_{\lambda/\lambda\cap\mu}(i)-m_{\lambda/\lambda\cap\mu}(j)}}{(1-tq^{2i}T)^{m_{\lambda/\lambda\cap\mu}(i)-m_{\lambda/\lambda\cap\mu}(j)}}.
        \end{equation*}
        Using $W(\lambda/\lambda\cap\mu,t)=W(\mu/\lambda\cap\mu,t)$, with both expressions reduced and unequal to $1$, we can equate the denominators. By knowing that there exists $i\in D(\lambda/\lambda\cap\mu)$, we would have that $1-tq^{2i}T$ divides the denominator of $W(\mu/\lambda\cap\mu,t)$ and thus should be equal to one of the irreducible factors appearing there. But since $q$ is not a root of unity, this would mean that $i\in D(\mu/\lambda\cap\mu)$. This is impossible, since 
        \[
        D(\lambda/\lambda\cap\mu)\cap D(\mu/\lambda\cap\mu)=\emptyset.
        \]
        Hence, $W(\lambda/\lambda\cap\mu,t)=1=W(\mu/\lambda\cap\mu,t)$, meaning that \[
        W(\lambda,t)=W(\lambda\cap\mu,t)=W(\mu,t).
        \]
         Since $\lambda\cap\mu\subset \mu,\lambda$, \cref{7.12} completes the proof.
    \end{proof}
    The following result follows immediately from the previous lemmas.
    \begin{thm}
        For $q\in\C$ generic, $t=q^{2a}$ with $a$ such that $B_n(q,t)$ is not semisimple and $(\lambda,f),(\mu,l)\in \Lambda_n$ it holds  $(\lambda,f)\overset{b}\sim(\mu,l)$, if and only if $W(\lambda,t)=W(\mu,t)$.
    \end{thm}
    \section{Description of primitive idempotents}
    \par In this section we give a description of the primitive idempotents in $B_n(q,t)$ in terms of the JM elements and the branching graph, for $q,t\in\C$ generic.
    \par Let $L(\lambda)$ be the simple corresponding to the partition $\lambda\in \Lambda_n$. By the Artin-Wedderburn theorem, it holds that 
    \begin{equation*}
        B_n(q,t)\overset{\varphi}{\cong} \bigoplus_{\lambda\in \Lambda_n}\End(L(\lambda))
    \end{equation*}
    as algebras over $\C$. Regarding $B_n(q,t)$ as a module over itself with the left regular action and $\End(L(\lambda))$ as a module over $B_n(q,t)$, the action being given by $(a.f)(x)=a.(f(x))$ makes $\varphi$ an isomorphism of representations.
    \par Since $L(\lambda)$ has a GZ basis $\{v_T \; | \; T\in T^{ud}(\lambda)\}$ which is completely determined by the corresponding eigenvalues of the family $x_1,\ldots,x_n$ acting on this basis, we easily get that 
    \[
    \prod_{T\in T^{ud}(\lambda)}\prod_{i=1}^n (x_i-c_T(i))=0
    \]
    as an operator acting on $L(\lambda)$. That is, following all the information of the branching graph up to $\lambda$ enables one to erase $L(\lambda)$ entirely from the Artin-Wedderburn decomposition of $B_n(q,t)$.
    \par This means that, if one wants to isolate the information of a specific $L(\mu)$ for $\mu\in\Lambda_n$, one can consider the element
    \begin{equation*}
        a_{\mu}=\prod_{{\mu\neq\lambda\in\Lambda_n}} \; \prod_{T\in T^{ud}(\lambda)} \;\prod_{i=1}^n (x_i-c_T(i))
    \end{equation*}
    and see immediately that for any $\lambda\neq \mu$, the action of $a_{\mu}$ on $L(\lambda)$ is zero, so that
    \[
    B_n(q,t)a_{\mu}\hookrightarrow \End(L(\mu)).
    \]
    \par The above formulas are highly ineffective and they still have not distinguished a summand of $\End(L(\mu))$ in its decomposition as a $B_n(q,t)$ module, $\End(L(\mu))=L(\mu)^{\oplus \dim L(\mu)}$. Thus, each of the paths to $\mu$ in the branching graph is able to produce a primitive idempotent corresponding to each one of the summands of $\End(L(\mu))$.
    \subsection{The primitive idempotents.}
    \par One can work inductively and fix the approach described in the previous subsection to produce a formula for the primitive idempotent $e_{\lambda}$ corresponding to the simple $L(\lambda)$ as follows.
    \par Pick $\lambda\in \Lambda_n$ and let $T^{d}(\lambda)=(\lambda_1=\emptyset\rightarrow \lambda_2\rightarrow \ldots \rightarrow \lambda_n=\lambda)$ be the drunk path to $\lambda$ and $\mu=\lambda_{n-1}$. Suppose $e_{\mu,n-1}$ is the primitive idempotent corresponding to $e_{\mu,n-1}$ with respect to the drunk path to $\mu$. Notice that the drunk path to $\mu$ is equal to $T^{d}(\lambda)$ when we truncate the last step. Now $e_{\mu,n-1}$ contains all the information of the branching graph up to $\mu$, hence the only information we need is the $x_n$ action on the GZ basis. Since we want to distinguish a summand, we "add" to $e_{\mu,n-1}$ the $x_n$ information corresponding to all $\lambda\neq\nu\in \Lambda_n$ and all of their paths, together with the information of all paths to $\lambda$ with the exception of the drunk path $T^{d}(\lambda)$. In other words, we conside the element 
    \begin{equation*}
       A= e_{\mu,n-1}\cdot \left( \prod_{\lambda\neq\nu\in{\Lambda_n}} \prod_{T\in T^{ud}(\nu)} \left(\frac{x_n-c_T(n)}{c_{T^{d}(\lambda)}(n)-c_T(n)}\right)\right)\prod_{\substack{S\in T^{ud}(\lambda), \\ S\neq T^{d}(\lambda)}}\left(\frac{x_n-c_S(n)}{c_{T^{d}(\lambda)}(n)-c_S(n)}\right).
    \end{equation*}
    Acting with $A$ on any GZ-basis element of the simple module $L(\nu)$ is equal to zero for any $\nu\neq \lambda$, since if $v_T$ is the basis element corresponding to $T\in T^{ud}(\nu)$, then $A$ has a factor $x_n-c_T(n)$, hence
    \[
    A.v_T=0.
    \]
    \par Now for the action on $L(\lambda)$. Order the paths in $T^{ud}(\lambda)$ in any way such that $T^{d}(\lambda)<T$ for any $T\neq T^{d}(\lambda)$ to get an ordered basis of $L(\lambda)$. If $T\neq T^{d}(\lambda)$, then $A$ again contains a factor 
    \[
    x_n-c_T(n),
    \]
    and thus $A.v_T=0$. The only action that matters is the one on $v_{T^{d}(\lambda)}$, for which the factors of $A$ act as follows. Since $e_{\mu,n-1}$ is the primitive idempotent corresponding to $\mu$ and the drunk path to $\mu$ is the path corresponding to the truncation of the last step in $T^{d}(\lambda)$, we have that 
    \[
    e_{\mu,n-1}v_{T^{d}(\lambda)}=v_{T^{d}(\lambda)}.
    \]
    Since all the other terms clearly evaluate to $1$, we have that
    \[
    A. v_{T^{d}(\lambda)}=v_{T^{d}(\lambda)}.
    \]
    This means that the image of $A$ under the Artin-Wedderburn isomorphism is equal to the block matrix, with block corresponding to the summand $\End(L(\lambda))$ equal to the elementary matrix $E_{11}$. Since $E_{11}$ is a primitive idempotent of $\End(L(\lambda))$ projecting to $L(\lambda)$, we get that 
    \begin{equation*}
        B_n(q,t)A\cong L(\lambda)
    \end{equation*}
    as $B_n(q,t)$ representations, where $B_n(q,t)A$ is the left ideal generated by $A$. Hence, we can state
    \begin{thm}
        Let $\lambda\in \Lambda_n, \; T^{d}(\lambda)$ be the drunk path to $\lambda$ and $\mu$ be the partition corresponding to the $n-1$ step in the path $T^{d}(\lambda)$. If $e_{\mu,n-1}$ is the primitive idempotent corresponding to $L(\mu)$ with respect to $T^{d}(\mu)$, then
        \begin{equation*}
            e_{\lambda,n}=e_{\mu,n-1}\cdot \left( \prod_{\lambda\neq\nu\in{\Lambda_n}} \prod_{T\in T^{ud}(\nu)} \left(\frac{x_n-c_T(n)}{c_{T^{d}(\lambda)}(n)-c_T(n)}\right)\right)\prod_{\substack{S\in T^{ud}(\lambda), \\ S\neq T^{d}(\lambda)}}\left(\frac{x_n-c_S(n)}{c_{T^{d}(\lambda)}(n)-c_S(n)}\right)
        \end{equation*}
        is the primitive idempotent corresponding to the simple $L(\lambda)$ with respect to $T^d(\lambda)$.
    \end{thm}
    Setting $e_{\emptyset,0}=1$, the above statement gives a description of a complete set of primitive idempotents in $B_n(q,t)$, in terms of the JM elements in the generic case.
    \subsection*{Acknowledgements}
    I would like to express my gratitude to my supervisor, Dr. Thorsten Heidersdorf, for the valuable discussions and for proofreading this manuscript. His feedback and suggestions were essential in making this version available.
\nocite{ResTur}
\nocite{Hu}
\nocite{Tur}
\nocite{GH}
\nocite{BBidem}
\nocite{EnyangGoodmanMurphybasis}
\bibliographystyle{amsrefs}
\bibliography{ref}

@article{Tur,
  author  = {Turaev, Vladimir G.},
  title   = {Operator invariants of tangles, and {$R$}-matrices},
  journal = {Math. USSR-Izv.},
  volume  = {35},
  year    = {1990},
  number  = {2},
  pages   = {411--?}
}

@misc{Mor,
  author = {Morton, Hugh R.},
  title  = {A basis for the {B}irman--{W}enzl algebra},
  year   = {2010},
  note   = {arXiv:1012.3116 [math.QA]}
}

@article{GH,
  author  = {Goodman, Frederick M. and Hauschild, Holly},
  title   = {Affine {B}irman--{W}enzl--{M}urakami algebras and tangles in the solid torus},
  journal = {Fund. Math.},
  volume  = {190},
  year    = {2006},
  pages   = {77--137}
}

@article{OV,
  author  = {Okounkov, Andrei and Vershik, Anatoly},
  title   = {A new approach to the representation theory of symmetric groups},
  journal = {Selecta Math. (N.S.)},
  volume  = {2},
  year    = {1996},
  number  = {4},
  pages   = {581--605}
}

@article{RS,
  author  = {Rui, Hebing and Song, Linliang},
  title   = {Decomposition matrices of {B}irman--{M}urakami--{W}enzl algebras},
  journal = {J. Algebra},
  volume  = {444},
  year    = {2015},
  pages   = {246--271}
}

@article{RSi2,
  author  = {Rui, Hebing and Si, Mei},
  title   = {Blocks of {B}irman--{M}urakami--{W}enzl algebras},
  journal = {Int. Math. Res. Not. IMRN},
  year    = {2011},
  number  = {2},
  pages   = {452--486}
}

@article{RSi1,
  author  = {Rui, Hebing and Si, Mei},
  title   = {Gram determinants and semisimplicity criteria for {B}irman--{W}enzl algebras},
  journal = {J. Reine Angew. Math.},
  volume  = {631},
  year    = {2009},
  pages   = {153--179}
}

@article{DRV,
  author  = {Daugherty, Zajj and Ram, Arun and Virk, Rahbar},
  title   = {Affine and degenerate affine {BMW} algebras: the center},
  journal = {Osaka J. Math.},
  volume  = {51},
  year    = {2014},
  number  = {1},
  pages   = {257--283}
}

@article{CRE,
  author  = {Creedon, Samuel},
  title   = {The center of the partition algebra},
  journal = {J. Algebra},
  volume  = {570},
  year    = {2021},
  pages   = {215--266}
}

@article{WEN,
  author  = {Wenzl, Hans},
  title   = {Quantum groups and subfactors of type {$B$}, {$C$}, and {$D$}},
  journal = {Comm. Math. Phys.},
  volume  = {133},
  year    = {1990},
  number  = {2},
  pages   = {383--432}
}

@article{MYSA,
  author  = {Mousaaid, Youssef and Savage, Alistair},
  title   = {Affinization of monoidal categories},
  journal = {J. \'Ec. polytech. Math.},
  volume  = {8},
  year    = {2021},
  pages   = {791--829}
}

@article{Naz,
  author  = {Nazarov, Maxim},
  title   = {Young's orthogonal form for {B}rauer's centralizer algebra},
  journal = {J. Algebra},
  volume  = {182},
  year    = {1996},
  number  = {3},
  pages   = {664--693}
}

@article{Hu,
  author  = {Hu, Jun},
  title   = {{BMW} algebra, quantized coordinate algebra and type {$C$} {S}chur--{W}eyl duality},
  journal = {Represent. Theory},
  volume  = {15},
  year    = {2011},
  pages   = {1--62}
}

@article{Xi,
  author  = {Xi, Changchang},
  title   = {On the quasi-heredity of {B}irman--{W}enzl algebras},
  journal = {Adv. Math.},
  volume  = {154},
  year    = {2000},
  number  = {2},
  pages   = {280--298}
}

@article{GL,
  author  = {Graham, J. J. and Lehrer, G. I.},
  title   = {Cellular algebras},
  journal = {Invent. Math.},
  volume  = {123},
  year    = {1996},
  number  = {1},
  pages   = {1--34}
}

@article{ResTur,
  author  = {Reshetikhin, N. and Turaev, V. G.},
  title   = {Invariants of {$3$}-manifolds via link polynomials and quantum groups},
  journal = {Invent. Math.},
  volume  = {103},
  year    = {1991},
  number  = {3},
  pages   = {547--597}
}

@article{Brauer,
  author  = {Brauer, Richard},
  title   = {On algebras which are connected with the semisimple continuous groups},
  journal = {Ann. of Math. (2)},
  volume  = {38},
  year    = {1937},
  number  = {4},
  pages   = {857--872}
}

@article{JungKim,
  author  = {Jung, Ji Hye and Kim, Myungho},
  title   = {Supersymmetric polynomials and the center of the walled {B}rauer algebra},
  journal = {Algebr. Represent. Theory},
  volume  = {23},
  year    = {2020},
  number  = {5},
  pages   = {1945--1975}
}

@article{EnyangGoodmanMurphybasis,
  author  = {Enyang, John and Goodman, Frederick M.},
  title   = {Cellular bases for algebras with a {J}ones basic construction},
  journal = {Algebr. Represent. Theory},
  volume  = {20},
  year    = {2017},
  number  = {1},
  pages   = {71--121}
}

@article{BirmanWenzl,
  author  = {Birman, Joan S. and Wenzl, Hans},
  title   = {Braids, link polynomials and a new algebra},
  journal = {Trans. Amer. Math. Soc.},
  volume  = {313},
  year    = {1989},
  number  = {1},
  pages   = {249--273}
}

@article{EnyangMurphyBasisold,
  author  = {Enyang, John},
  title   = {Cellular bases for the {B}rauer and {B}irman--{M}urakami--{W}enzl algebras},
  journal = {J. Algebra},
  volume  = {281},
  year    = {2004},
  number  = {2},
  pages   = {413--449}
}

@article{sspolys,
  author  = {Stembridge, John R.},
  title   = {A characterization of supersymmetric polynomials},
  journal = {J. Algebra},
  volume  = {95},
  year    = {1985},
  number  = {2},
  pages   = {439--444}
}

@article{JucysCenterCSn,
  author  = {Jucys, A.-A. A.},
  title   = {Symmetric polynomials and the center of the symmetric group ring},
  journal = {Rep. Math. Phys.},
  volume  = {5},
  year    = {1974},
  number  = {1},
  pages   = {107--112}
}

@article{HaRamPartition,
  author  = {Halverson, Tom and Ram, Arun},
  title   = {Partition algebras},
  journal = {European J. Combin.},
  volume  = {26},
  year    = {2005},
  number  = {6},
  pages   = {869--921}
}

@article{MR927059,
  author  = {Murakami, Jun},
  title   = {The {K}auffman polynomial of links and representation theory},
  journal = {Osaka J. Math.},
  volume  = {24},
  year    = {1987},
  number  = {4},
  pages   = {745--758}
}

@book{MR255,
  author    = {Weyl, Hermann},
  title     = {The {C}lassical {G}roups. Their {I}nvariants and {R}epresentations},
  publisher = {Princeton University Press},
  address   = {Princeton, NJ},
  year      = {1939}
}

@article{BGCMHT,
  author  = {Benkart, Georgia and Chakrabarti, Manish and Halverson, Thomas and Leduc, Robert and Lee, Chanyoung and Stroomer, Jeffrey},
  title   = {Tensor product representations of general linear groups and their connections with {B}rauer algebras},
  journal = {J. Algebra},
  volume  = {166},
  year    = {1994},
  number  = {3},
  pages   = {529--567}
}

@article{MR1265453,
  author  = {Martin, Paul},
  title   = {Temperley--{L}ieb algebras for nonplanar statistical mechanics---the partition algebra construction},
  journal = {J. Knot Theory Ramifications},
  volume  = {3},
  year    = {1994},
  number  = {1},
  pages   = {51--82}
}

@article{LZ,
  author  = {Lehrer, G. I. and Zhang, R. B.},
  title   = {The {B}rauer category and invariant theory},
  journal = {J. Eur. Math. Soc. (JEMS)},
  volume  = {17},
  year    = {2015},
  number  = {9},
  pages   = {2311--2351}
}

@misc{chavli2024centerwalledbraueralgebra,
  author = {Chavli, Eirini and De Visscher, Maud and Parker, Alison and Salmon, Sarah and Wilson, Ulrica},
  title  = {The center of the walled {B}rauer algebra {$B_{r,1}(\delta)$}},
  year   = {2024},
  note   = {arXiv:2412.14716 [math.RT]}
}

@article{IsaevOgievetsky,
  author  = {Isaev, A. P. and Ogievetsky, O. V.},
  title   = {Jucys--{M}urphy elements for {B}irman--{M}urakami--{W}enzl algebras},
  journal = {Phys. Part. Nucl. Lett.},
  volume  = {8},
  year    = {2011},
  number  = {3},
  pages   = {234--243}
}

@article {BBidem,
    AUTHOR = {Beliakova, Anna and Blanchet, Christian},
     TITLE = {Skein construction of idempotents in
              {B}irman-{M}urakami-{W}enzl algebras},
   JOURNAL = {Math. Ann.},
  FJOURNAL = {Mathematische Annalen},
    VOLUME = {321},
      YEAR = {2001},
    NUMBER = {2},
     PAGES = {347--373},
      ISSN = {0025-5831,1432-1807},
   MRCLASS = {57M27 (16S99 46L65)},
  MRNUMBER = {1866492},
MRREVIEWER = {Richard\ M.\ Green},
       DOI = {10.1007/s002080100233},
       URL = {https://doi.org/10.1007/s002080100233},
}

%\printbibliography

\end{document}